\documentclass[oneside,a4paper,english, 11pt]{amsart}

\input{artmymacros.sty}

\title{The weight part of Serre's conjecture over CM fields}

\author{Daniel Le}
\address{Department of Mathematics,
Purdue University,
150 N. University Street, 
West Lafayette, IN 47907-2067}
\email{ledt@purdue.edu}

\author{Bao V.~Le Hung}
\address{Department of Mathematics,
Northwestern University, 
2033 Sheridan Road, 
Evanston, Illinois 60208, USA}
\email{lhvietbao@googlemail.com}

\begin{document}

\begin{abstract} 
Under some technical assumptions of a global nature, we establish the weight part of Serre's conjecture for mod $p$ Galois representations for CM fields that are tamely ramified and sufficiently generic at $p$. 
\end{abstract}

\maketitle

\section{Introduction}

\subsection{The weight part of Serre's conjecture}
In \cite{serre-duke}, Serre made a bold conjecture asserting that every continuous, odd, and irreducible representation $\rbar: G_{\Q}\defeq \Gal(\ovl{\Q}/\Q) \ra \GL_2(\ovl{\F}_p)$ arises from a modular form. 
Further, he spelled out a strong form of conjecture predicting the minimal weight and level for the modular form in terms of the local properties of $\rbar$. 
These conjectures initiated a mod $p$ and $p$-adic Langlands philosophy that would play an influential role in understanding reciprocity between Galois representations and automorphic forms. 

A decade later, Ash, Doud, Pollack, and Sinnott \cite{Ash,AS00,ADP} generalized Serre's modularity conjecture to arbitrary dimension. 
First, every Hecke eigenclass in cohomology of a mod $p$ local system on an arithmetic locally symmetric space $Y(\Gamma) \defeq \Gamma \backslash \GL_n(\R)/\mathrm{O}_n(\R)\R^\times$ should give rise to a continuous and odd Galois representation $\rbar: G_{\Q} \ra \GL_n(\ovl{\F}_p)$. 
Second, every continuous, odd, and irreducible representation $\rbar: G_{\Q} \ra \GL_n(\ovl{\F}_p)$ should arise in this way. 
These conjectures go beyond Langlands' original reciprocity conjecture, as when $n>2$, the cohomology of $Y(\Gamma)$ tends to have abundant torsion which has no direct connection to automorphic forms. 

Early work on Serre's original conjecture focused on showing that the weak form implies the strong form. 
More than a decade later, Khare and Wintenberger \cite{KW} were able to use this implication to execute an elaborate inductive strategy to prove the weak form of Serre's conjecture. 
Investigations into generalizations of Serre's conjecture have to this point followed a similar shape. 
As in Serre's original conjecture, the minimal level giving rise to $\rbar$ should be $\Gamma_1(N)$ for $N$ the prime-to-$p$ Artin conductor of $\rbar$. 
In contrast, the generalization of the weight part of Serre's conjecture turns out to be far more subtle. Beyond the simplest setting of holomorphic modular forms, it turns out to be much more robust to formulate the weight part representation theoretically, in terms of simple $\ovl{\F}_p[\GL_n(\F_p)]$-modules called \emph{Serre weights}. The Serre weights give rise to tautological mod $p$ local systems on $Y(\Gamma)$, and the weight part of Serre's conjecture is equivalent to a description of the set $W(\rbar)$ of Serre weights $\sigma$ for which there exists Hecke eigenclasses in $H^*(Y(\Gamma),\sigma)$ that give rise to $\rbar$. 
The most basic expectation is that $W(\rbar)$ depends only on the local representation $\rhobar \defeq \rbar|_{G_{\Q_p}}$. 

Serre gave an elaborate but explicit description of $W(\rbar)$ in terms of $\rhobar$ in dimension $2$, and later \cite{BDJ,schein-IL} generalized Serre's recipe to totally real extensions of $\Q$. 
\cite{ADP} made some predictions about $W(\rbar)$ in higher rank, but at present, there are no complete generalizations of the weight part of Serre's conjecture in any dimension greater than $2$. 
An explicit recipe in arbitrary dimension is perhaps too much to ask for. 
One can interpret the $\GL_2$-recipes in terms of the geometry of the recently constructed Emerton--Gee moduli stacks of mod $p$ local Galois representations whose irreducible components $\cC_\sigma$ are indexed by Serre weights $\sigma$. 
It is this geometric interpretation that should generalize---the complexity of a general recipe reflects (in part) the complexity of the geometry of these stacks. 

In another direction, Herzig studied the case when $\rbar$ is tamely ramified at $p$ \cite{herzig-duke}. 
This case is interesting for at least two reasons. 
First, this should be the richest case and $W(\rbar)$ is the largest. 
Second, tame inertial representations can be classified combinatorially, and so one might hope for a combinatorial formula for $W(\rbar)$. 
Herzig defined a set $W^?(\rhobar)$ for tamely ramified $\rhobar$ in terms of the decomposition of the reduction mod $p$ of a $\GL_n(\F_p)$-representation naturally associated to the restriction of $\rhobar$ to the inertial subgroup. 
Then he predicted that $W^?(\rhobar)$ should at least contain the set of $p$-regular Serre weights in $W(\rbar)$. 
For generic $\rhobar$, the set $W^?(\rhobar)$ indeed admits a combinatorial description. 

\subsection{Results}

In recent years, there has been a flurry of results on the weight part of Serre's conjecture. 
Gee and his collaborators have obtained essentially complete results in the context of $2$-dimensional representations of $G_F$ where $F/\Q$ is a totally real extension \cite{GK,GLS}. 
There has also been progress in higher rank. 
Let  $W^g(\rhobar)$ be the set of Serre weights $\sigma$ for which $\rhobar\in \cC_\sigma$. 
The authors with Levin and Morra \cite{MLM} have shown that for definite unitary groups, $W(\rbar) = W^?(\rhobar) = W^g(\rhobar)$ when the combinatorial parameter for the restriction of $\rhobar$ to inertia is \emph{super generic}, in the sense it avoids certain universal polynomial congruences. 
These advances combine the Taylor--Wiles method with analyses of local Galois deformation rings. 

Still, to this point, there had been no progress on the weight part of Serre's conjecture beyond dimension $2$ in the context in which Ash, Herzig, and others originally proposed their conjectures, or more generally the context of locally symmetric spaces of $\GL_n$ over a number field $F$. There are several reasons to expect this context to be substantially harder:
\begin{itemize}
\item The relevant locally symmetric spaces are not algebraic varieties, hence there is no easy bridge to connect automorphic forms and Galois representations.
\item It has been observed (e.g. in \cite{BV}) that the (integral) cohomology of the relevant locally symmetric space consists predominantly of torsion, and in some sense most Hecke eigenclasses are not liftable to characteristic $0$ and thus have no direct automorphic interpretation. 
\item The invariant $\ell_0$ from \cite{BorelWallach} which controls the range of cohomology of locally symmetric spaces where tempered automorphic forms contribute is positive. In particular, the tempered part of cohomology does not concentrate in one degree, which breaks the key numerical coincidence that makes the Taylor--Wiles patching process work.
\end{itemize}
 
However, our understanding of this situation has dramatically improved over the last decade. First, when $F$ is a CM or totally real field, Scholze \cite{scholze} constructed the expected Galois representations attached to possibly torsion Hecke eigenclasses, refining earlier results for characteristic $0$ classes by Harris--Lan--Taylor--Thorne \cite{HLTT}, thus overcoming the first two difficulties.
Second, Calegari--Geraghty \cite{CG} found a way to modify the Taylor--Wiles patching process in contexts where cohomology does not concentrate, though its full power is conditional on several conjectures pertaining to structural information on torsion in cohomology as well as their associated Galois representations.

In this paper, we prove the analogue of the aforementioned result of \cite{MLM} for $\GL_n$ over an imaginary CM field $F$ in super generic situations (with additional technical assumptions), thus obtaining the first results on Herzig's formulation of the weight part of Serre's conjecture in a setting where $\ell_0>0$.
 
\begin{thm}\label{thm:main intro}\emph{(see Theorem \ref{thm:main})}
Let $F$ be a CM field containing an imaginary quadratic field. 
Assume that 
\begin{itemize}
\item $[F:\Q]\geq 10$; and
\item $p > 2n+1$ is a prime that splits completely in $F$. 
\end{itemize}

Let $\rbar: G_F \ra \GL_n(\F)$ be the mod $p$ reduction of the Galois representation associated to a cuspidal automorphic representation $\pi$ of $\GL_n(\A_F)$ such that
\begin{itemize}
\item $\pi$ has weight $0$; 
\item $\pi$ has Iwahori fixed vectors away from $p$; and
\item $\pi^{\ker(\GL_n(\cO_F)\ra \GL_n(\cO_F/p))} \neq 0$. 
\end{itemize} 
Assume that 
\begin{itemize}
\item $\rbar$ is decomposed generic in the sense of \cite[Definition 4.3.1]{10author}; 
\item $\rbar|_{G_{F(\zeta_p)}}$ is absolutely irreducible; and
\item $\rbar(G_F - G_{F(\zeta_p)})$ contains a scalar matrix. 
\end{itemize}

Finally, assume that $\rbar|_{G_{F_w}}$ is is tamely ramified and super generic for all places $w|p$. 
Then $W(\rbar)$ is the set 
\[
\big\{\bigotimes_{w|p} \sigma_w \mid \sigma_w \in W^?(\rbar|_{G_{F_w}})\big\}
\] 
predicted by Herzig. 
\end{thm}

\begin{rmk} \mbox{}
\begin{enumerate}
\item In the main text, the hypothesis on the CM field $F$ is weaker but more artificial. 
It (and the decomposed genericity of $\rbar$) arises only due to the need to invoke the local-global compatibility results of \cite{Hevesi}. In fact, we have written our arguments in an axiomatic framework, so that Theorem \ref{thm:main intro} immediately generalizes to the case $F$ is totally real (and in particular to the setting $F=\Q$ originally considered by Ash, Herzig and others) whenever the corresponding generalization of \cite{Hevesi} is obtained.
\item The assumption that $\rbar$ arises as the mod $p$ reduction of an automorphic Galois representation is key to our method for obtaining the lower bound on $W(\rbar)$, though it is not needed for the upper bound of $W(\rbar)$. 
It is not clear to us how restrictive this assumption is. 
On the one hand, experience has shown that mod $p$ \emph{Hecke eigenclasses} usually do not lift to characteristic $0$ eigenclasses when $\ell_0>0$. 
On the other hand, there seems to be no clear consensus on whether the \emph{system of Hecke eigenvalues} lifts to characteristic $0$ (see the discussion in \cite[\S 1]{FKP}), which is what we actually need. 
However, there are plenty of examples where Theorem \ref{thm:main intro} applies via base change of automorphic forms coming from unitary groups. 
\item The technical conditions on the images of $\rbar$ are standard hypotheses for Taylor--Wiles patching, which are therefore essential to our arguments.
\item The hypothesis that $\rbar_w$ is tame and super generic means that the set of inertial weights of $\rbar_w$ avoids certain universal polynomial congruences mod $p$, and is hard to make explicit. Its presence is due to the fact that we invoke the local results of \cite{MLM}. It is conceivable that our method could relax this genericity to the more familiar condition that asks the inertial weight to avoid certain hyperplanes mod $p$, though we have not yet succeeded in doing so.  
\end{enumerate}
\end{rmk}

\subsection{Methods}
Our method is based on a modification of the techniques of \cite{MLM} which is robust enough to apply to the $\ell_0>0$ setting. To lighten the burden of notation, we will pretend $F=\Q$.
Following Calegari--Geraghty, we can ``patch'' the cohomology complexes to produce a perfect complex $C_\infty$ of $S_\infty[\GL_n(\F_p)]$-modules for a certain formally smooth $\Z_p$-algebra $S_\infty$, while (morally) at the same time having compatible actions of local Galois deformation rings. This has the key feature that certain specializations of $C_\infty$ recovers the cohomology of $Y(\Gamma)$ with coefficients in tautological local systems; in particular, $W(\rbar)$ is precisely the set of $\sigma$ for which $C_\infty \otimes_{\GL_n(\F_p)} \sigma \neq 0$.
Further, local-global compatibility for Galois representations attached to Hecke eigenclasses implies that for a lattice $L$ in a Deligne--Lusztig representation of $\GL_n(\F_p)$, the support of the tensor product $C_\infty \otimes_{\GL_n(\F_p)} L$ is contained in a potentially semistable deformation ring defined in terms of the Deligne--Lusztig representation and inertial local Langlands. 
These properties in fact constitute the axiomatic setup to which our arguments apply. 

The method from \cite{MLM} in the $\ell_0=0$ setting proceeds in the following steps: 
\begin{enumerate}
\item (support bound/weight elimination) Bound the support cycle of $C_\infty \otimes_{\GL_n(\F_p)} \sigma$. In particular, this shows that $W(\rbar) \subset W^?(\rhobar)$.
\item (modularity of obvious weights) Show that $W(\rbar)$ contains a special collection $W_\obv(\rhobar)$ of \emph{obvious} weights. 
\item (full support) Show that for any lattice $L$ in a Deligne--Lusztig representation, the support cycle of $C_\infty \otimes_{\GL_n(\F_p)} L/p$ is a constant multiple of the special fiber cycle of the corresponding potentially semistable deformation ring. 
By varying $L$, we compute the support cycle of  $C_\infty \otimes_{\GL_n(\F_p)} \sigma$, in particular showing that it is non-zero exactly when $\sigma\in W^?(\rhobar)$. 
\end{enumerate}
When $\ell_0=0$, $C_\infty$ is a projective $S_\infty[\GL_n(\F_p)]$-module, so that 
whenever $\sigma$ is a Jordan--H\"older factor of $\ovl{L}\defeq L/p$, the difference of the support cycle of $C_\infty \otimes_{\GL_n(\F_p)} \ovl{L}$ and $C_\infty \otimes_{\GL_n(\F_p)} \sigma$ is effective. This positivity underlies the key fact that if the latter is non-zero, then so is the former.
This key fact combined with a theory of (na\"ive) local models for potentially crystalline Galois deformation rings immediately gives (1).
Using step (1), the modularity of obvious weights can then be proved using potential diagonalizability change of type techniques from \cite{BLGGT}. 
Full support requires two further observations: 
\begin{itemize}
\item the reduction of any $L$ for which the corresponding potentially semistable Galois deformation ring is nonzero contains an element of $W_\obv(\rhobar)$; and
\item any potentially semistable Galois deformation ring of minimal regular weight is either an integral domain or $0$ since local models are unibranch at a super generic and tame $\rhobar$ \cite{MLM}. 
\end{itemize}
The first observation combined with the key fact and step (2) shows that the cycle of $C_\infty \otimes_{\GL_n(\F_p)} \ovl{L}$ is non-zero, and then the second observation shows that this cycle must be a positive multiple of the special fiber cycle of the corresponding deformation ring.

We now turn to the situation $\ell_0>0$. Under the Calegari--Geraghty vanishing conjecture \cite[Conjecture B(4)]{CG},  $C_\infty$ would again concentrate in one degree, and the above argument carries over verbatim. Unfortunately, the vanishing conjecture seems out of reach at the moment, so we have to proceed differently.
The key difficulty we have to face is that $C_\infty \otimes_{\GL_n(\F_p)} \sigma$ may not concentrate, and hence its support cycle need not be effective (for example, it might reduce to $0$ even when $\sigma\in W^?(\rhobar)$). This causes the key fact (and thus all the arguments above) to completely break down. Additionally, for step (2) we no longer have access to the potential diagonalizability change of type arguments.

We now explain the new ideas needed to execute the above program without the key fact.
While the weight elimination argument of \cite{MLM,LLL} does not immediately show that $C_\infty \otimes_{\GL_n(\F_p)} \sigma = 0$ for $\sigma\notin W^?(\rhobar)$, it does produce a degree shift in the possible non-vanishing cohomology group (for a different $\sigma'$). 
Knowledge of the interaction of $W^?(\rhobar)$ with Deligne--Lusztig representations and induction on degree allows us to show the desired weight elimination. 
In fact, the same argument after localizing at generic points allows us to show the desired support bound. In particular, we get the expected upper bound $W(\rbar)\subset W^?(\rhobar)$.

It remains to produce the matching lower bound $W(\rbar)\supset W^?(\rhobar)$, that is, one needs to produce enough elements in $W(\rbar)$. 
In \cite{OBW}, we bypass the use of potential diagonalizability for proving the modularity of obvious weights by relying instead on an $\ell = p$ analogue of Taylor's ``Ihara avoidance" method \cite{taylor}. 
A derived version of Ihara avoidance was recently developed and applied to the $\ell_0 > 0$ setting in \cite{10author} to great effect. Roughly speaking, this method spreads the existence of a component $\cC_\sigma$ in the support cycle of $C_\infty \otimes_{\GL_n(\F_p)} \sigma$ to the existence of a component $\cC_{\sigma'}$ in the support cycle of $C_\infty \otimes_{\GL_n(\F_p)} \sigma'$, as long as $\cC_{\sigma}, \cC_{\sigma'}$ simultaneously appear in a well-understood part of a common ``connecting" potentially semistable stack. Our main observation is that one can enlarge the set of types considered in \cite{OBW} to connect not just pairs of obvious weights $\sigma,\sigma'\in W_\obv(\rhobar)$, but in fact to connect any pair of weights in the much larger $W^?(\rhobar)$. This allows us to spread membership in $W(\rbar)$ throughout $W^?(\rhobar)$, as long as one has a starting member (this is where we invoke the existence of the automorphic lift of $\rbar$). 
Finally, we remark that even when specializing to the case $\ell_0=0$, our arguments give a new, more robust proof of the main result on the weight part of Serre's conjecture in \cite{MLM}. 

\subsection{Ackowledgments}
D.L.~was supported by the National Science Foundation under agreement DMS-2302623 and a start-up grant from Purdue University. 
B.LH.~acknowledges support from the National Science Foundation under grants Nos.~DMS-1952678 and DMS-2302619 and the Alfred P.~Sloan Foundation. 

\subsection{Notation}

Let $n$ be a positive integer. 
Let $p$ be a prime. 
Fix an algebraic closure $\ovl{\Q}_p$ of $\Q_p$. 
Let $E \subset \ovl{\Q}_p$ be a finite extension of $\Q_p$ with ring of integers $\cO\subset \ovl{\Z}_p$, uniformizer $\varpi$, and residue field $\F \defeq \cO/\varpi$. 
We will assume that $E$ is sufficiently large as specified by the context.

\section{Preliminaries}

\subsection{Affine Weyl group notations}\label{sec:affine Weyl notation}
Let $G$ be a split reductive group with a maximal torus $T \subset B$ in a Borel subgroup. 
Recall the following standard notations \cite[\S 1.3]{DLR}:
\begin{itemize}
\item the character group $X^*(T)$ of $T$;
\begin{itemize}
\item $R\subset X^*(T)$ the set of roots of $G$ with respect to $T$; 
\item for $\nu \in X^*(T)$ or $X^*(T)\otimes_{\Z} \R$, let $h_\nu \defeq \max_{\alpha \in R} \langle \nu,\alpha^\vee\rangle$; 
\item $X^0(T) \subset X^*(T)$ the set of elements $\nu$ with $\langle \nu,\alpha^\vee\rangle = 0$ for all $\alpha \in R$, i.e.~$h_\nu = 0$; 
\item $R^+ \subset R$ the subset of positive roots with respect to $B$, i.e.~the roots occurring in $\mathrm{Lie}(B)$; note that this is the convention in \cite{jantzen} but opposite to \cite{RAGS}; this determines a partial ordering $\leq$ on $X^*(T)$ characterized by $\lambda \leq \mu$ if and only if $\mu-\lambda$ is a sum (possibly with multiplicities) of elements in $R^+$; 
\item $\Delta \subset R^+$ the subset of simple roots; 
\item for each $\alpha \in \Delta$, we fix a choice of $\omega_\alpha \in X^*(T)$ (unique up to $X^0(T)$) with $\langle \omega_\alpha,\beta^\vee\rangle = \delta_{\alpha\beta}$ for any $\beta\in \Delta$; 
\item $X(T)^+\subset X^*(T)$ the dominant weights with respect to $R^+$; 
\item the $p$-restricted set $X_1(T)\subset X(T)^+$ of dominant weights $\lambda$ such that $\langle \lambda,\alpha^\vee\rangle \leq p-1$ for all $\alpha \in \Delta$; 
\item let $\eta = \sum_{\alpha \in \Delta} \omega_\alpha \in X^*(T)$ so that $\langle \eta,\alpha^\vee\rangle = 1$ for all $\alpha\in\Delta$;
\item 
\end{itemize}
\item the Weyl group $W$ of $(G,T)$ and $w_0\in W$ its longest element;
\begin{itemize}
\item the extended affine Weyl group $\tld{W} \defeq X^*(T) \rtimes W$, which acts on $X^*(T)$ on the left by affine transformations; for $\nu\in X^*(T)$ we write $t_\nu\in \tld{W}$ for the corresponding element;
\item the affine Weyl group $W_a \defeq \Z R \rtimes W \subset \tld{W}$; 
\item For a root $\alpha \in R$, let $s_\alpha\in W$ be the corresponding reflection. 
\end{itemize}
\item the set of alcoves of $X^*(T)\otimes_{\Z}\R$, i.e.~the set of connected components of 
\[
X^*(T)\otimes_{\Z}\R\setminus\bigcup_{n\in\Z,\alpha\in R}\{\lambda\in X^*(T)\otimes_{\Z}\R\mid \langle \lambda,\alpha^\vee\rangle=n\},
\]
which has a (transitive) left action of $\tld{W}$; 
\begin{itemize}
\item the dominant alcoves, i.e.~alcoves $A$ such that $0<\langle \lambda,\alpha^\vee\rangle$ for all $\alpha\in \Delta, \lambda\in A$;
\item the lowest (dominant) alcove $A_0 = \{\lambda \in X^*(T)\otimes_{\Z}\R\mid 0<\langle \lambda,\alpha^\vee\rangle<1 \textrm{ for all } \alpha\in R^+\}$; 
\item $\Omega\subset \tld{W}$ the stabilizer of the base alcove;
\item the restricted alcoves, i.e.~alcoves $A$ such that $0<\langle \lambda,\alpha^\vee\rangle<1$ for all $\alpha\in \Delta, \lambda\in A$; note that the union of the restricted alcoves form a fundamental domain for the translation action of $X^*(T)$ on the set of all alcoves; 
\item the set $\tld{W}^+\subset \tld{W}$ of elements $\tld{w}$ such that $\tld{w}(A_0)$ is dominant;
\item the set $\tld{W}_1\subset \tld{W}^+$ of elements $\tld{w}$ such that $\tld{w}(A_0)$ is restricted; 
\item $\tld{w}_h = w_0t_{-\eta} \in \tld{W}_1$; 
\item for $\tld{w} \in \tld{W}$, let $\tld{w}^\diamond$ denote an element in $X^*(T) \tld{w} \cap \tld{W}_1$ (which exists since the union of the restricted alcoves form a fundamental domain for the translation action of $X^*(T)$ on the set of all alcoves); it is unique up to $X^0(T)$. 
\end{itemize}
\item the set of $p$-alcoves of $X^*(T)\otimes_{\Z}\R$, i.e.~the set of connected components of 
\[
X^*(T)\otimes_{\Z}\R\setminus\bigcup_{n\in\Z,\alpha\in R}\{\lambda\in X^*(T)\otimes_{\Z}\R\mid \langle \lambda+\eta,\alpha^\vee\rangle=np\};
\]
\begin{itemize}
\item a left $p$-dot action of $\tld{W}$ on $X^*(T)$ defined by $(t_\nu w)\cdot \lambda\defeq p\nu+w(\lambda+\eta)-\eta$; this induces a $p$-dot action of $\tld{W}$ on the set of $p$-alcoves whose restriction to $W_a$ is simply transitive;
\item the dominant $p$-alcoves, i.e.~alcoves $C$ such that $0<\langle \lambda+\eta,\alpha^\vee\rangle$ for all $\alpha\in \Delta, \lambda\in C$;
\item the lowest (dominant) $p$-alcove $C_0\subset X^*(T)\otimes_{\Z}\R$ characterized by $\lambda\in C_0$ if $0<\langle \lambda+\eta,\alpha^\vee\rangle<p$ for all $\alpha\in R^+$;
\item the $p$-restricted alcoves, i.e.~alcoves $C$ such that $0<\langle \lambda+\eta,\alpha^\vee\rangle<p$ for all $\alpha\in \Delta, \lambda\in C$;
\end{itemize}
\item the Bruhat order $\leq$ on $W_a$ with respect to $A_0$ (i.e.~using the reflections across walls of $A_0$ as a set of Coxeter generators);
\begin{itemize}
\item for $\tld{w} \in \tld{W}$, let $\tld{W}_{\leq \tld{w}} = \{\tld{u}\in \tld{W} \mid \tld{u} \leq \tld{w}\}$; 
\item for $\lambda \in X^*(T)$, let $\Adm(\lambda) = \cup_{w\in W} \tld{W}_{\leq t_{w(\lambda)}}$;
\item the $\uparrow$ order on the set of $p$-alcoves defined in \cite[II.6.5]{RAGS};
\item the $\uparrow$ order on $W_a$ induced from the ordering $\uparrow$ on the set of $p$-alcoves (via the bijection $\tld{w}\mapsto\tld{w}\cdot C_0$);
\item the Bruhat order on $\tld{W}=W_a\rtimes\Omega$ defined by $\tld{w}\delta\leq \tld{w}'\delta'$ if and only if $\tld{w}\leq \tld{w}'$ and $\delta = \delta'$ where $\delta,\delta'\in\Omega$ and $\tld{w},\tld{w}'\in W_a$;
\item the $\uparrow$ order on $\tld{W}$ defined by $\tld{w}\delta\uparrow \tld{w}'\delta'$ if and only if $\tld{w}\uparrow \tld{w}'$ and $\delta = \delta'$ where $\delta,\delta'\in\Omega$ and $\tld{w},\tld{w}'\in W_a$;
\end{itemize}
\end{itemize}
We will assume throughout that $h_\eta < p$ so that $C_0$ is nonempty. 
\subsection{Combinatorial lemmas}

\begin{lemma}\label{lemma:reduced1}
Let $\tld{w}_2 \in \tld{W}_1$ and $\tld{w}_1 \in \tld{W}^+$ such that $\tld{w}_1 \uparrow \tld{w}_h^{-1} \tld{w}_2$. Let $\alpha$ be a simple root. Then the factorization $\tld{w}_2^{-1} (s_\alpha w_0) \tld{w}_1$ is reduced.
\end{lemma}
\begin{proof}
Choose minimal galleries $G_1$ and $G_2$ to $\tld{w}_1(A_0)$ and $\tld{w}_2(A_0)$ from the dominant base alcove \cite[\S 2]{HC}.
These are galleries in the $\id$-direction \cite[Definition 5.2 and Lemma 5.3]{HC}. 
Since the length of a minimal gallery coincides with the number of separating hyperplanes, a minimal gallery cannot cross a single hyperplane twice. 
As $\tld{w}_2\in \tld{W}_1$, $\tld{w}_2(A_0)$ and $A_0$ are in the same $\alpha$-strip ($\langle x,\alpha^\vee \rangle \in (0,1)$ and $\langle \tld{w}_2(x),\alpha^\vee \rangle \in (0,1)$ for any $x\in A_0$) and so the minimality of $G_2$ implies that $G_2$ does not cross any $\alpha$-hyperplanes. 
In particular, $G_2$ is also a gallery in the $s_{\alpha}$-direction because any separating hyperplane $H$ is defined by a positive root $\beta \neq \alpha$ so that $s_{\alpha}(\beta)>0$. 
Choose a minimal gallery $G_3$ (necessarily in the $\id$-direction) from $w_0 (A_0)$ to $s_\alpha (A_0)$.
Then the concatenation of the reverse of $w_0 (G_1)$ followed by $G_3$ and then $s_\alpha(G_2)$, all galleries in the $\id$-direction, is a gallery in the $\id$-direction from $w_0 \tld{w}_1(A_0)$ to $s_\alpha\tld{w}_2(A_0)$ and hence is minimal \cite[Lemma 5.3]{HC}.
Applying $\tld{w}_2^{-1}s_\alpha$ to this minimal gallery, we see that \[\ell(\tld{w}_2^{-1} s_\alpha w_0 \tld{w}_1) = \ell(\tld{w}_2^{-1})+ \ell(s_\alpha w_0) +\ell(\tld{w}_1).\]
\end{proof}

\begin{lemma}\label{lemma:omega}
Let $\tld{w}_2 \in \tld{W}_1$, $\alpha \in \Delta$, and $\omega \in X^*(T)$ be such that $(s_\alpha \tld{w}_2)^\diamond = t_\omega s_\alpha\tld{w}_2 \in \tld{W}_1$. 
Then 
\begin{enumerate}
\item \label{item:omegaalpha} $\langle \omega,\alpha^\vee\rangle = 1$; 
\item \label{item:omegabeta} if $\beta$ is a simple root orthogonal to $\alpha$, then $\langle \omega,\beta^\vee\rangle = 0$; 
\item \label{item:omegagamma} if $\beta$ is a simple root with $\langle \beta,\alpha^\vee \rangle \leq 0$, then $\langle \omega,\beta^\vee\rangle \leq 0$; and
\item \label{item:omegagamma'} if $\gamma$ is a positive root and $\langle \omega_\alpha,\gamma^\vee\rangle \leq 1$, then $\langle \omega,\gamma^\vee\rangle \leq 1$. 
\end{enumerate} 
\end{lemma}
\begin{proof}
Let $x$ be in $A_0$ and $\beta$ be a simple root. 
Then 
\[
\langle (s_\alpha \tld{w}_2)^\diamond (x),\beta^\vee \rangle - \langle \omega,\beta^\vee\rangle = \langle s_\alpha\tld{w}_2 (x),\beta^\vee \rangle = \langle \tld{w}_2(x),s_\alpha(\beta^\vee)\rangle = \langle \tld{w}_2(x),\beta^\vee-\langle \beta,\alpha^\vee\rangle\alpha^\vee\rangle
\]
and $\langle (s_\alpha\tld{w}_2)^\diamond(x),\beta^\vee\rangle \in (0,1)$ since $(s_\alpha\tld{w}_2)^\diamond \in \tld{W}_1$ and $\beta\in \Delta$.  

When $\beta = \alpha$, then $\langle s_\alpha\tld{w}_2 (x),\alpha^\vee \rangle = \langle \tld{w}_2 (x),-\alpha^\vee \rangle \in (-1,0)$ and hence $\langle \omega,\alpha^\vee \rangle \in (0,2)$, so that \eqref{item:omegaalpha} follows. 

When $\langle \beta,\alpha^\vee\rangle = 0$, then $\langle s_\alpha\tld{w}_2 (x),\beta^\vee \rangle= \langle \tld{w}_2 (x),\beta^\vee \rangle \in (0,1)$ (recall that $\tld{w}_2\in \tld{W}_1$), and hence $\langle \omega,\alpha^\vee \rangle \in (-1,1)$ so that \eqref{item:omegabeta} follows. 

When $\langle \beta,\alpha^\vee\rangle \leq 0$, then 
\[
\langle (s_\alpha\tld{w}_2)^\diamond (x),\beta^\vee \rangle - \langle \omega,\beta^\vee\rangle = \langle \tld{w}_2(x),\beta^\vee-\langle \beta,\alpha^\vee\rangle\alpha^\vee\rangle \geq \langle \tld{w}_2(x),\beta^\vee\rangle, 
\]
and hence $\langle \omega,\alpha^\vee \rangle <1$ so that \eqref{item:omegagamma} follows. 

Now let $\gamma$ be a root, so that $\gamma = \sum_{\beta \in\Delta} \langle \omega_\beta,\gamma^\vee\rangle \beta$. 
If $\langle \omega_\alpha ,\gamma^\vee \rangle \leq 1$ and $\gamma > 0$, then \eqref{item:omegagamma'} follows from \eqref{item:omegaalpha} and \eqref{item:omegagamma}. 
\end{proof}

\begin{lemma}\label{lemma:reduced2}
Let $\alpha$ be a simple root such that $h_{\omega_\alpha} = 1$. 
If $\tld{w}_2 \in \tld{W}_1$, then
\[(s_\alpha\tld{w}_2)^{\diamond,-1}((s_\alpha\tld{w}_2)^\diamond \tld{w}_2^{-1} s_\alpha w_0)\] is a reduced factorization of $\tld{w}_2^{-1} s_\alpha w_0$.
\end{lemma}
\begin{proof}
Let $\omega \in X^*(T)$ be such that $(s_\alpha\tld{w}_2)^\diamond = t_\omega s_\alpha \tld{w}_2$.
We now prove the lemma by computing lengths, starting with $(s_\alpha\tld{w}_2)^\diamond \tld{w}_2^{-1} s_\alpha w_0 = t_\omega w_0$.
Let $x\in A_0$ and $\beta$ be a positive root. 
Since $\langle \omega_\alpha,\beta^\vee\rangle\leq h_{\omega_\alpha} = 1$, Lemma \ref{lemma:omega}\eqref{item:omegagamma'} (with $\gamma = \beta$) implies that $\langle \omega,\beta^\vee\rangle \leq 1$. 
For $\tld{w} \in \tld{W}$ and $\beta \in R^+$, let $m_\beta(\tld{w})$ denote the number of $\beta$-hyperplanes separating $A_0$ and $\tld{w}(A_0)$ i.e.~$m_\beta(\tld{w}) = |\min_{x\in \ovl{A}_0} \langle \tld{w}(x),\beta^\vee \rangle|$. 
Then 
\begin{align*}
m_\beta(t_{\omega}w_0) &= |\min_{x\in \ovl{A}_0} \langle t_\omega w_0(x),\beta^\vee \rangle| \\
&= |\langle \omega,\beta^\vee\rangle + \min_{x\in \ovl{A}_0} \langle w_0(x),\beta^\vee \rangle| \\
&= |\langle \omega,\beta^\vee\rangle -1| \\
&= 1-\langle \omega,\beta^\vee\rangle 
\end{align*}
since $\langle \omega,\beta^\vee\rangle \leq 1$.
We have  
\begin{align*}
\ell(t_\omega w_0) &= \sum_{\beta > 0} m_\beta(t_{\omega}w_0) \\
&= \sum_{\beta > 0} (1-\langle  \omega,\beta^\vee\rangle)\\
&= \ell(w_0) - \sum_{\beta > 0} \langle  \omega,\beta^\vee\rangle. 
\end{align*}

We now compute the length of $(s_\alpha\tld{w}_2)^\diamond$.
For $\beta>0$ and $\beta \neq \alpha$ and $x\in A_0$, $\langle s_\alpha \tld{w}_2(x),\beta^\vee \rangle = \langle \tld{w}_2(x),s_\alpha(\beta)^\vee \rangle > 0$ and $\langle (s_\alpha\tld{w}_2)^\diamond(x),\beta^\vee \rangle > 0$ (since $\beta, s_\alpha(\beta) > 0$ and $\tld{w}_2,(s_\alpha\tld{w}_2)^\diamond\in \tld{W}_1$).
This implies that $m_\beta((s_\alpha\tld{w}_2)^\diamond) = m_\beta(s_\alpha\tld{w}_2) + \langle  \omega,\beta^\vee\rangle$.
On the other hand, using \ref{lemma:omega}\eqref{item:omegaalpha} we have 
\[
m_\alpha((s_\alpha\tld{w}_2)^\diamond) = 0 = \langle  \omega,\alpha^\vee \rangle - 1 = m_\alpha(s_\alpha \tld{w}_2) + \langle  \omega,\alpha^\vee \rangle - 2, 
\]
where the first equality is a consequence of the fact that $(s_\alpha\tld{w}_2)^\diamond \in \tld{W}_1$ and $\alpha$ is simple. 
Putting this together, we get
\begin{align*}
\ell((s_\alpha\tld{w}_2)^\diamond) &= \sum_{\beta > 0} m_\beta((s_\alpha\tld{w}_2)^\diamond) \\
&= -2+ \sum_{\beta > 0} (m_\beta(s_\alpha\tld{w}_2) + \langle \omega,\beta^\vee\rangle) \\
&=  \ell(s_\alpha \tld{w}_2) + \sum_{\beta > 0} \langle  \omega,\beta^\vee\rangle - 2 \\
&= \ell(\tld{w}_2) + \sum_{\beta > 0} \langle  \omega,\beta^\vee\rangle - 1 
\end{align*}
(note that $\ell(s_\alpha\tld{w}_2) = \ell(\tld{w}_2) + 1$ by \cite[Lemma 2.2.1]{OBW}).
Putting everything together, we have that
\[
\ell((s_\alpha\tld{w}_2)^\diamond \tld{w}_2^{-1} s_\alpha w_0) + \ell((s_\alpha\tld{w}_2)^{\diamond,-1}) = \ell(t_\omega w_0) + \ell((s_\alpha\tld{w}_2)^\diamond) = \ell(w_0)+ \ell(\tld{w}_2) - 1 = \ell(\tld{w}_2^{-1} s_\alpha w_0),
\]
where the last equality follows for example from Lemma \ref{lemma:reduced1}
\end{proof}

\begin{lemma}\label{lem:subregular shape bound}
Let $\alpha$ be a simple root such that $\langle \omega_\alpha,\beta^\vee\rangle \leq 1$ for all roots $\beta$. 
Let $\tld{w}_2 \in \tld{W}^+_1$ and $\tld{w}_1 \in \tld{W}^+$ such that $\tld{w}_1 \uparrow \tld{w}_h^{-1} \tld{w}_2$. 
Then $(s_\alpha\tld{w}_2)^\diamond \tld{w}_2^{-1} s_\alpha w_0 \tld{w}_1 \leq w_0\tld{w}_h^{-1} (s_\alpha\tld{w}_2)^\diamond$. 
\end{lemma}
\begin{proof}
By Lemmas \ref{lemma:reduced1} and \ref{lemma:reduced2}, we have that
\[
(s_\alpha\tld{w}_2)^{\diamond,-1}((s_\alpha\tld{w}_2)^\diamond \tld{w}_2^{-1} s_\alpha w_0)\tld{w}_1
\]
is a reduced factorization.
Then it suffices to show that 
\[
\tld{w}_2^{-1} s_\alpha w_0 \tld{w}_1 = (s_\alpha\tld{w}_2)^{\diamond,-1}((s_\alpha\tld{w}_2)^\diamond \tld{w}_2^{-1} s_\alpha w_0)\tld{w}_1 \leq (s_\alpha\tld{w}_2)^{\diamond,-1} w_0 \tld{w}_h^{-1} (s_\alpha\tld{w}_2)^\diamond 
\]
(use \cite[Lemma 4.1.9]{LLL} to get a reduced factorization of the right-most expression). 
We claim that  
\[
\tld{w}_2^{-1} s_\alpha w_0 \tld{w}_1 \leq \tld{w}_2^{-1} s_\alpha w_0 \tld{w}_h^{-1} \tld{w}_2 = (s_\alpha\tld{w}_2)^{\diamond,-1} s_\alpha w_0 \tld{w}_h^{-1} (s_\alpha\tld{w}_2)^\diamond \leq (s_\alpha\tld{w}_2)^{\diamond,-1} w_0 \tld{w}_h^{-1} (s_\alpha\tld{w}_2)^\diamond. 
\]
Indeed, we have the following. 
\begin{itemize}
\item The first inequality follows from reduced factorizations obtained by Lemma \ref{lemma:reduced1} and the inequality $\tld{w}_1 \leq \tld{w}_h^{-1} \tld{w}_2$ from \cite[Theorem 4.3]{Wang} (see also \cite[Theorem 4.1.1]{LLL}).
\item The last inequality follows from reduced factorizations obtained from \cite[Lemma 4.1.9]{LLL}. 
\item For the equality in the middle, let $\omega \in X^*(T)$ be such that $(s_\alpha\tld{w}_2)^\diamond = t_\omega s_\alpha\tld{w}_2$. Noting that $w_0\tld{w}_h^{-1} = t_{w_0\eta}$ and $s_\alpha(w_0\eta+\omega)=w_0\eta+\omega$ (since $\langle \omega, \alpha^\vee \rangle=-\langle w_0\eta,\alpha^\vee\rangle=1$ by Lemma \ref{lemma:omega}\eqref{item:omegaalpha} for the first equality and by definition of $\eta$ for the second), we write
\begin{align*}
(s_\alpha\tld{w}_2)^{\diamond,-1} s_\alpha w_0 \tld{w}_h^{-1} (s_\alpha\tld{w}_2)^\diamond&=\tld{w}_2^{-1}s_\alpha t_{-\omega} s_\alpha t_{w_0\eta}t_{\omega}s_\alpha \tld{w}_2=\\\tld{w}_2^{-1}s_\alpha t_{-\omega+s_\alpha(w_0\eta+\omega)} \tld{w}_2&=\tld{w}_2^{-1}s_\alpha  t_{w_0\eta}\tld{w}_2=\tld{w}_2^{-1} s_\alpha w_0 \tld{w}_h^{-1} \tld{w}_2  
\end{align*}
thus establishing the desired equality.
\end{itemize}
\end{proof}

\subsection{Serre weights}

Let $G_0$ be a connected reductive group over $\F_p$ which splits over the finite extension $\F$. 
Let $G$ be $G_0 \otimes_{\F_p} \F$, which we assume to have simply connected derived subgroup and connected center. 
Fix a maximal torus and Borel subgroup $T\subset B\subset G$. 
For $\lambda \in X(T)^+$,  there is a unique simple $\F[G]$-module $L(\lambda)$ with highest weight $\lambda$. 

We say that $\lambda \in X^*(T)$ is $m$-deep in its $p$-alcove if $|\langle \lambda+\eta,\alpha^\vee\rangle-np| >m$ for all $\alpha \in R$ and $n\in \Z$. 

Let $F$ be the relative $p$-Frobenius isogeny of $G$, which induces an endomorphism $F$ of $X^*(T)$. 
There is an automorphism $\pi$ of $X^*(T)$ such that $F=p\pi^{-1}$ on $X^*(T)$. 
We assume that the choice of $\eta \in X^*(T)$ is $\pi$-invariant. 

Let $\Gamma$ be $G_0(\F_p)$. 
A Serre weight (for $\Gamma$) is a simple $\F[\Gamma]$-module. 
If $\lambda \in X_1(T)$, let $F(\lambda) \defeq L(\lambda)|_\Gamma$. 
Then $F(\lambda)$ is a Serre weight and any Serre weight is of this form. 
Moreover, $F(\lambda) \cong F(\mu)$ if and only if $\lambda-\mu \in (p-\pi)X^0(T)$. 
We say that $F(\lambda)$ is $m$-deep in $p$-alcove $C$ if $\lambda$ is (and this is independent of the choice of $\lambda$). 

If $\lambda \in X_1(T)$ is in alcove $\tld{w} \cdot C_0$, let $d_{F(\lambda)} \defeq \max_{v\in \ovl{A}_0} h_{\tld{w}_h\tld{w}(v)}$. 
Note that $d_{F(\lambda)}\leq h_\eta$. 
\begin{defn} (Presentation of Serre weight)

For $\omega-\eta\in C_0\cap X^*(T)$ and $\tld{w}\in\tld{W}_1$, let
\begin{equation}
F_{(\tld{w},\omega)}\defeq F(\pi^{-1}(\tld{w})\cdot (\omega-\eta)).
\end{equation}
\end{defn}
We say that $(\tld{w},\omega)$ is a presentation of a Serre weight $\sigma$ if 
$\sigma\cong F_{(\tld{w},\omega)}$. 
We say that $\sigma$ is $m$-deep if it admits a presentation $(\tld{w},\omega)$ where $\omega-\eta$ is $m$-deep in $C_0$. 
Note that not all Serre weights are $0$-deep and thus may not have presentations in the above sense. 

\subsection{Deligne--Lusztig representations}

For $(s,\mu) \in W\times X^*(T)$, one defines an $F$-stable maximal torus $T_s$ and a character $\Theta(s,\mu)$ of $\Gamma \cap T_s$ as in \cite[\S 3]{DLR}. 
Deligne--Lusztig induction yields a virtual $\Gamma$-representation $R_s(\mu)$ over $E$ (again see \cite[\S 3]{DLR}). In this paper, we will also abbreviate $R_s(\mu)=R(t_\mu s)$ and hence make sense of the notation $R(\tld{w})$ for $\tld{w}\in \tld{W}$. 
We call either the choice of $(s,\mu)$ or $\tld{w}$ a presentation of $R(\tld{w})$. 
We warn the reader that this notion of presentation differs by an $\eta$-shift from e.g.~\cite{LLL,MLM} and agrees with \cite{FLH}. 

We will also use the notation $\ovl{R}(\tld{w})$ to denote the mod $\varpi$ reduction of any $\cO$-lattice in $R(\tld{w})$. This is not a well defined $\F[\Gamma]$-module, but its class in the Grothendieck group of finite length $\F[\Gamma]$-module is well defined, and hence invariants such as Jordan-H\"older factors are well-defined.

We say that $(s,\mu)$ is maximally split if the pair $(T_s,\Theta(s,\mu))$ is. 
If $(s,\mu)$ is maximally split, then $R_s(\mu)$ is an actual representation. 
The main examples of such maximally split $(s,\mu)$ that we need are when either $s = 1$ or $\mu-\eta$ is $0$-deep in $C_0$.
If $R \cong R(\tld{w})$ such that $\tld{w}(0)-\eta$ is $m$-deep in $C_0$, we say that $R$ is an $m$-generic Deligne--Lusztig representation. 

\begin{lemma}\label{lemma:0gen}
Recall that we assume that $G$ has connected center. 
Suppose that $\mu-\eta,\lambda-\eta \in C_0$, $\mu-\lambda \in \Z R$, $s,w\in W$, and $R_s(\mu) \cong R_w(\lambda)$. 
Then $\mu = \lambda$ and $s = w$. 
\end{lemma}
\begin{proof}
By \cite[\S 4.1]{herzig-duke}, $w = \sigma s \pi(\sigma)^{-1}$ and $\lambda = \sigma(\mu) + p\nu - w \pi \nu$ for some $\sigma \in W$ and $\nu \in X^*(T)$ (using that $F$ in \emph{loc.~cit}~is $p \pi^{-1}$ in our notation). 
\begin{claim} \label{claim:rootlattice}
We have that $\nu\in \Z R$. 
\end{claim}
\begin{proof}
Since $\mu-\lambda\in \Z R$, we deduce that $(p-\pi)\nu\in \Z R$. 
This implies that $(p^f - 1)\nu\in \Z R$, where $f$ is the order of $\pi$. 
If $Z$ is the (connected) center of $G$, then the torsion subgroup of $X^*(Z)$ is a $p$-group. 
Using that $X^*(T)/\Z R \cong X^*(Z)$, the claim follows. 
\end{proof}

\begin{claim}\label{claim:ht2}
We have that $\langle \nu,\alpha^\vee\rangle \leq 2$ with equality if and only if 
\begin{itemize}
\item $\alpha^\vee$ is a highest coroot; 
\item $\langle \mu,\alpha^\vee\rangle = p-1 = \langle \lambda,\alpha^\vee\rangle$; 
\item $\sigma$ commutes with $s_\alpha$;  
\item $\langle w(\pi(\nu)),\alpha^\vee \rangle =2$. 
\end{itemize}
\end{claim}
\begin{proof}
Suppose that $\alpha\in R$ such that $\langle \nu,\alpha^\vee\rangle = h_\nu$. 
Then using that $\lambda = \sigma(\mu) + p\nu - w(\pi(\nu))\in C_0 + \eta$, we have 
\begin{align*}
p-1 &\geq \langle \lambda,\alpha^\vee \rangle \\
&= \langle \sigma (\mu),\alpha^\vee\rangle + p\langle \nu,\alpha^\vee\rangle - \langle w(\pi(\nu)),\alpha^\vee \rangle \\
&\geq 1-p+(p-1)h_\nu,
\end{align*}
from which we deduce that $h_\nu \leq 2$. 
Moreover, if $\langle \nu,\alpha^\vee\rangle = 2$, then all inequalities above must be equalities. 
In particular, $p-1 = \langle \lambda,\alpha^\vee \rangle$. 
Since $\langle \lambda,\beta^\vee\rangle > 0$ for any $\beta \in \Delta$, $\alpha^\vee$ must be a highest coroot. 
Next, $1-p = \langle \sigma (\mu),\alpha^\vee\rangle = \langle \mu,\sigma^{-1}(\alpha)^\vee \rangle$ so that $\sigma^{-1}(\alpha) = \alpha$ and $\langle \mu,\alpha^\vee \rangle = p-1$. 
That $\sigma^{-1}(\alpha) = \alpha$ implies that $\sigma$ and $s_\alpha$ commute (look at the actions on the set of roots). 
Finally, we have that $\langle w(\pi(\nu)),\alpha^\vee \rangle =2$. 
The converse is easy to check. 
\end{proof}

\begin{claim}\label{claim:not1}
Fix $\alpha \in R$ and $\kappa\in \Z R$. 
If $0<\max_{r\in W} \langle \kappa,r(\alpha)^\vee\rangle \leq 2$ then $\max_{r\in W} \langle \nu,r(\alpha)^\vee\rangle = 2$. 
\end{claim}
\begin{proof}
Indeed, if $\max_{r\in W} \langle \kappa,r(\alpha)^\vee\rangle \leq 1$, then the projection of $\kappa$ to the root system spanned by $\{r(\alpha)\mid r\in W$ is contained in $W \ovl{A}_0$. 
However, $W \ovl{A}_0 \cap \Z R = 0$ so that $\max_{r\in W} \langle \nu,r(\alpha)^\vee\rangle = 0$ which is a contradiction. 
\end{proof}

Combining Claims \ref{claim:ht2} and \ref{claim:not1}, we see that $\max_{r\in W} \langle \nu,r(\alpha)^\vee\rangle>0$ with $\nu$ as above implies that $\max_{r\in W} \langle \kappa,r(\alpha)^\vee\rangle = 2$. 

\begin{claim} \label{claim:dom}
The weight $\nu$ is dominant. 
\end{claim}
\begin{proof}
Let $\alpha \in \Delta$ such that $\langle \nu,\alpha^\vee\rangle \neq 0$. 
Then Claim \ref{claim:not1} implies that $\langle \nu,r(\alpha)^\vee\rangle =2$ for some $r\in W$ so that by Claim \ref{claim:ht2}, $\beta^\vee \defeq r(\alpha)^\vee$ is a highest coroot and $s_\beta$ commutes with $\sigma$. 
We have that 
\[
\langle \sigma (\mu),\alpha^\vee\rangle + p\langle \nu,\alpha^\vee\rangle - \langle w(\pi(\nu)),\alpha^\vee \rangle = \langle \lambda,\alpha^\vee\rangle > 0. 
\]
If $\langle \nu,\alpha^\vee\rangle < 0$, then using that $h_\nu \leq 2$ we have $\langle \sigma(\mu),\alpha^\vee \rangle > p-2$. 
This implies that 
\begin{itemize}
\item $\langle \sigma (\mu),\alpha^\vee\rangle=\langle \mu,\sigma^{-1}(\alpha)^\vee \rangle = p-1$; 
\item $\sigma^{-1}(\alpha)^\vee$ is a highest coroot and thus equal to $\beta^\vee$; and
\item $\langle \nu,\alpha^\vee \rangle = -1$. 
\end{itemize}
The fact that $\sigma$ and $s_\beta$ commute implies that $\alpha = \sigma(\beta) = -\beta$. 
Thus $\langle \nu,\alpha^\vee\rangle = -\langle \nu,\beta^\vee\rangle = -2$ which contradicts the last bullet point. 
\end{proof}

\begin{claim}\label{claim:highestcoroot}
If $\kappa \in \Z R$ is such that 
\begin{itemize}
\item $\langle \kappa,\alpha^\vee \rangle \leq 2$ for all $\alpha\in R$ and equality implies that $\alpha^\vee$ is a highest coroot; and
\item $\kappa$ is dominant; 
\end{itemize}
then $\kappa$ is a multiplicity free sum of roots whose corresponding coroot is highest. 
\end{claim}
\begin{proof}
We immediately reduce to the case where the root system is irreducible. 
The rank one case is easy to check directly, so we now assume that the root system has rank at least $2$. 
If $\max_{\alpha\in R} \langle \kappa,\alpha^\vee\rangle = 0$, then $\kappa = 0$, and we are done. 
Otherwise, $\max_{\alpha\in R} \langle \kappa,\alpha^\vee\rangle > 0$, and Claim \ref{claim:not1} implies that $\max_{\alpha\in R} \langle \kappa,\alpha^\vee\rangle =2$. 
As $\kappa$ is dominant (by the second bullet point), we have that $\langle \kappa,\alpha_0^\vee\rangle =2$ for the highest coroot $\alpha_0^\vee$. 
If $\alpha \in \Delta$ and $\langle \alpha_0,\alpha^\vee\rangle \neq 0$, then $\alpha_0^\vee-\alpha^\vee$ is a coroot but not a highest coroot. 
By the first bullet point, $\langle \kappa,\alpha_0^\vee-\alpha^\vee \rangle <2$ so that $\langle \kappa,\alpha^\vee\rangle > 0$. 
If $\alpha_0 = \sum_{\alpha\in \Delta} n_\alpha \alpha$, then one can check casewise using the classification of irreducible root systems (and that the rank is at least $2$) that 
\begin{equation}\label{eqn:rootdecomp}
\sum_{\substack{\alpha \in \Delta\\ \langle \alpha_0,\alpha^\vee\rangle \neq 0}} n_\alpha = 2. 
\end{equation} 
Using \eqref{eqn:rootdecomp} and that $\langle \kappa,\alpha_0^\vee\rangle =2$, $\langle \kappa,\alpha^\vee\rangle \geq 1$ if $\alpha \in \Delta$ and $\langle \alpha_0,\alpha^\vee\rangle \neq 0$, and $\langle \kappa,\alpha^\vee\rangle \geq 0$ for all $\alpha \in \Delta$, we conclude that for $\alpha\in \Delta$, $\langle \kappa,\alpha^\vee \rangle = 1$ (resp.~$\langle \kappa,\alpha^\vee\rangle = 0$) if $\langle \alpha_0,\alpha^\vee\rangle \neq 0$ (resp.~$\langle\alpha_0,\alpha^\vee\rangle = 0$). 
Since the same formulas hold with $\alpha_0$ in place of $\kappa$, this implies that $\kappa = \alpha_0$. 
\end{proof}

Combining Claims \ref{claim:rootlattice}, \ref{claim:ht2}, \ref{claim:dom}, and \ref{claim:highestcoroot}, we see that $\nu$ is a multiplicity-free sum of roots $\alpha_i$ whose corresponding coroots $\alpha_i^\vee$ are highest coroots. 
Say $\nu = \sum_{i\in S} \alpha_i$. 

\begin{claim} \label{claim:nustable}
We have $w(\pi(\nu)) = \nu$. 
\end{claim}
\begin{proof}
Using that for any $i\in S$ 
\[
\langle \nu,\pi^{-1}(w^{-1}(\alpha_i))^\vee\rangle = \langle w(\pi(\nu)),\alpha_i^\vee\rangle = 2 
\]
from Claim \ref{claim:ht2}, another application of Claim \ref{claim:ht2} shows that $\pi^{-1}(w^{-1}(\alpha_i))^\vee$ must be a highest coroot as well. 
Then $\pi^{-1}(w^{-1}(\alpha_i))^\vee = \alpha_j^\vee$ for some $j\in S$ since  it is not orthogonal to $\nu$. 
In other words, for any $j\in S$, $w(\pi(\alpha_j)) = \alpha_i$ for some $i\in S$. 
The claim follows. 
\end{proof}

\begin{claim}\label{claim:sigma}
We have $\sigma = \prod_{i\in S} s_{\alpha_i}$. 
\end{claim}
\begin{proof}
Since $\sigma$ and $s_{\alpha_i}$ commute for each $i \in S$, $\sigma$ can be written as a product of the reflections $s_{\alpha_i}$ and of the reflections $s_\alpha$ with $\alpha \in R^+$ and $\langle\alpha_i,\alpha^\vee\rangle = 0$. 

If $\alpha \in R^+$ and $\langle\alpha_i,\alpha^\vee\rangle = 0$ for all $i\in S$, then 
\[
0<\langle \lambda,\alpha^\vee\rangle = \langle \mu,\sigma^{-1}(\alpha)^\vee\rangle + p\langle \nu,\alpha^\vee\rangle - \langle w(\pi(\nu)),\alpha^\vee\rangle = \langle \mu,\sigma^{-1}(\alpha)^\vee\rangle
\]
using Claim \ref{claim:nustable}. 
We deduce that $\sigma^{-1}(\alpha) \in R^+$. 
Thus, $\sigma$ is in the group generated by the reflections $\{s_{\alpha_i}\mid i\in S\}$. 

For $i\in S$, we have 
\[
p>\langle \lambda,\alpha_i^\vee\rangle = \langle \mu,\sigma^{-1}(\alpha_i)^\vee\rangle + p\langle \nu,\alpha_i)^\vee\rangle - \langle w(\pi(\nu)),\alpha_i^\vee\rangle =\langle \mu,\sigma^{-1}(\alpha_i)^\vee\rangle + 2(p-1),
\]
so that $\langle \mu,\sigma^{-1}(\alpha_i)^\vee\rangle < 0$. 
Then $-\sigma^{-1}(\alpha_i) \in R^+$ for each $i\in S$. 
We conclude that $\sigma = \prod_{i\in S} s_{\alpha_i}$. 
\end{proof}

Combining Claims \ref{claim:nustable} and \ref{claim:sigma}, we see that $\lambda = \mu$. 
We also have that 
\[
w \pi(\sigma) w^{-1} = \prod_{i\in S} w \pi(s_{\alpha_i}) w^{-1} = \prod_{i\in S}  s_{w(\pi(\alpha_i))} = \sigma
\]
where the first equality uses Claim \ref{claim:sigma} and the last equality follows from Claim \ref{claim:nustable} (and that $\nu = \sum_{i\in S} \alpha_i$). 
Then $s = \sigma ^{-1} w \pi(\sigma) =w$. 
\end{proof}

\begin{lemma}\label{lem:DL containing decent weight} 
Let $(s,\mu)$ be maximally split and $\sigma$ be an $(m+d_\sigma)$-deep Serre weight for some $m\geq 0$ such that $\sigma\in \JH(\ovl{R}_s(\mu))$. 
Assume that either $m > 2$ or $R_s(\mu)$ is $0$-generic. Then $R_s(\mu)$ is $m$-generic. 
\end{lemma}
\begin{proof} 
If $R$ is $0$-generic, then $R \cong R_{s'}(\mu')$ with $\mu'\in C_0$. 
Replacing $(s,\mu)$ by $(s',\mu')$ (which is also maximally split), we can and do assume in this case that $\mu \in C_0$. 
If $R$ is not $0$-generic, we instead assume without loss of generality that $\mu$ is dominant and $\langle \mu,\alpha^\vee \rangle \leq p+2$ for all roots $\alpha$ using \cite[Lemma 2.3.3]{MLM}. 
Suppose that $\sigma = F(\lambda)$. 
The fact that $\sigma \in \JH(\ovl{R}_s(\mu))$ implies that $\Hom(Q_\lambda,\ovl{R}_s(\mu)) \neq 0$ where $Q_\lambda$ is defined as in \cite[\S 3]{DLR}. 
By \cite[Theorem 3.2]{DLR}, there exists $\tld{w} \in \tld{W}$ such that $\tld{w}\cdot (\mu-\eta+s\pi \tld{w}^{-1}(0))+\eta$ is dominant and 
\begin{equation}\label{item:uparrow}
\tld{w}\cdot (\mu-\eta+s\pi \tld{w}^{-1}(0)) \uparrow \tld{w}_h\cdot \lambda. 
\end{equation}
In particular, we have $\tld{w}\cdot (\mu-\eta+s\pi \tld{w}^{-1}(0)) +\eta\leq  \tld{w}_h\cdot \lambda + \eta$. 
We claim that $h_{\tld{w}(0)} \leq d_\sigma$. 
Let $w \in W$ such that $w\tld{w} \in \tld{W}^+$ and let $\alpha_0^\vee$ be a highest coroot such that $h_{\tld{w}(0)} = \langle w\tld{w}(0),\alpha_0^\vee\rangle$. 
(We can take $\alpha_0^\vee$ to be a highest coroot since $w\tld{w}(0)$ is dominant.) 
In particular, $\alpha_0^\vee$ is dominant. 
Then we have inequalities 
\begin{align*}
(p-1)h_{\tld{w}(0)} - p - 2 &\leq \langle p w\tld{w}(0) +\sigma s \pi \tld{w}^{-1}(0) + \sigma(\mu),\alpha_0^\vee \rangle \\
&= \langle w\tld{w}\cdot (\mu-\eta +s\pi \tld{w}^{-1}(0))+\eta,\alpha_0^\vee \rangle \\
&\leq \langle \tld{w}\cdot (\mu-\eta +s\pi \tld{w}^{-1}(0))+\eta,\alpha_0^\vee \rangle \\
&\leq \langle \tld{w}_h\cdot \lambda+\eta,\alpha_0^\vee\rangle \\
&\leq (p-1)d_\sigma-1-m 
\end{align*}
where 
\begin{itemize} 
\item $\sigma$ denotes the projection of $w\tld{w}$ in $W$; 
\item the first inequality uses that $h_\mu \leq p+2$; 
\item the second equality is actually an equality obtained from an equality of the corresponding terms in the pairings; 
\item the third inequality uses that $\tld{w}\cdot (\mu-\eta+s\pi \tld{w}^{-1}(0))+\eta\geq w\tld{w}\cdot (\mu-\eta+s\pi \tld{w}^{-1}(0))+\eta$ since $\tld{w}\cdot (\mu-\eta+s\pi \tld{w}^{-1}(0))+\eta$ is dominant and also that $\alpha_0^\vee$ is dominant; and
\item the fourth inequality uses \eqref{item:uparrow} and that $\alpha_0^\vee$ is dominant. 
\end{itemize}
For the last inequality, suppose that $\lambda \in \tld{w} \cdot \nu$ for $\nu \in C_0$. 
Note that $\tld{w}_h\cdot \lambda + \eta = \tld{w}_h\tld{w} \cdot \nu +\eta = p\tld{w}_h\tld{w}(\frac{\nu+\eta}{p})$ and $\frac{\nu+\eta}{p} \in A_0$. 
By the definition of $d_\sigma$, we have $\langle \tld{w}_h\tld{w}(\frac{\nu+\eta}{p}),\alpha_0^\vee\rangle  \leq d_\sigma$ so that $\langle \tld{w}_h\cdot \lambda+\eta,\alpha_0^\vee\rangle \leq pd_\sigma$. 
Since $\lambda$ is $(m+d_\sigma)$-deep in its $p$-alcove (and thus so is $\tld{w}_h\cdot \lambda$), $\langle \tld{w}_h\cdot \lambda+\eta,\alpha_0^\vee\rangle < pd_\sigma - m - d_\sigma$ which gives the last inequality. 

Taking the sequence of inequalities and using that $m>2$, we get that $(p-1)h_{\tld{w}(0)} < (p-1)(d_\sigma +1)$ which implies the desired inequality. 
If instead, we have that $\mu \in C_0$, then the first expression in the sequence of inequalities can be replaced by $(p-1)h_{\tld{w}(0)} - p + 2$. 
Taking the modified sequence of inequalities and using that $m\geq 0$ similarly yields the desired inequality. 

Since $\mu-\eta+s\pi \tld{w}^{-1}(0)$ and $\lambda$ are in the same $\tld{W}$-orbit under the $p$-dot action, $\mu-\eta+s\pi \tld{w}^{-1}(0)$ is $(m+d_\sigma)$-deep in its $p$-alcove. 
The inequality in the previous paragraph then implies that $\mu-\eta$ is $m$-deep in its $p$-alcove. 
In the case that $R$ is $0$-generic, $\mu-\eta$ is in $C_0$ by assumption, and we are done. 
Otherwise, we are in the case that $m>2$. As $\mu$ is dominant and $h_\mu \leq p+2$, we must have that $\mu-\eta$ is $m$-deep in $C_0$. 
\end{proof}

\begin{rmk}\label{rmk:typeA}
In type $A$, we can replace $m>2$ in Lemma \ref{lem:DL containing decent weight} with $m>0$. 
Indeed, in the proof, we can assume without loss of generality that $h_\mu \leq p$ by \cite[Remark 2.3.5]{MLM}. 
\end{rmk}

We have the following useful description of the Jordan-H\"older factors of generic Deligne--Lusztig representations which strengthens \cite[Proposition 2.3.8]{MLM}:
\begin{prop}\label{prop:combinatorial membership of DL} Suppose we are given a Serre weight $\sigma$ and a Deligne--Lusztig representation $R=R(t_\mu s)$ such that $\mu-\eta$ is $h_\eta$-deep in $C_0$. Then the following are equivalent: 
\begin{enumerate}
\item \label{item:JH}
$\sigma\in \JH(\ovl{R})$; 
\item \label{item:uparrowreflect}
there exists $\tld{w}\in \tld{W}_1$, $\omega\in \eta+ C_0$, and $\tld{u}\in\tld{W}^+$ such that $\sigma=F_{(\tld{w},\omega)}$, $\tld{u}\uparrow \tld{w}_h\tld{w}$, and 
\[t_\omega\in t_\mu s \tld{u}^{-1}W;\]
and
\item \label{item:adm}
there exists $\tld{w}\in \tld{W}_1$, $\omega\in \eta+ C_0$ such that $\sigma=F_{(\tld{w},\omega)}$ and 
\[t_{\omega}\tld{W}_{\leq w_0\tld{w}}\subset t_\mu s \Adm(\eta).\]
\end{enumerate}
\end{prop}
\begin{proof} The equivalence between \eqref{item:JH} and \eqref{item:uparrowreflect} is a restatement of \cite[Theorem 4.2]{DLR} where $\tld{w}$, $\tld{w}_\lambda$, and $\mu$ in loc.~cit.~are taken to be $\pi^{-1}\tld{u}$, $\pi^{-1}\tld{w}$, and $\omega-s\tld{w}^{-1}(0)$, respectively. 
The equivalence between \eqref{item:uparrowreflect} and \eqref{item:adm} follows from the proof of \cite[Proposition 2.3.8]{MLM} which we recall for the convenience of the reader. 
\eqref{item:uparrowreflect} implies that 
\[
t_\omega \tld{W}_{\leq w_0\tld{w}} \subset t_\mu s \tld{u}^{-1} \tld{W}_{\leq w_0 \tld{w}} \subset t_\mu s \tld{W}_{\leq (\tld{w}_h\tld{w})^{-1}w_0 \tld{w}} \subset t_\mu s \Adm(\eta)
\] 
since $\tld{u}\leq \tld{w}_h\tld{w}$ by \cite[Theorem 4.1.1]{LLL} and $(\tld{w}_h\tld{w})^{-1}w_0 \tld{w}$ is a reduced factorization by \cite[Lemma 4.1.9]{LLL}. 
For the other direction, uniquely write $(t_\mu s)^{-1}t_\omega = \tld{u}^{-1}w$ for some $w\in W$ and $\tld{u} \in \tld{W}^+$. 
Then 
\[
\tld{u}^{-1}w_0 \tld{w} = \tld{u}^{-1} w (w^{-1}w_0 \tld{w}) \in  (t_\mu s)^{-1}t_\omega \tld{W}_{\leq w_0\tld{w}} \subset \Adm(\eta) 
\]
where the membership uses that $w_0 \tld{w}$ is a reduced factorization by \cite[Lemma 4.1.9]{LLL} and the inclusion is equivalent to \eqref{item:adm}. 
We are done by \cite[Proposition 2.1.6]{MLM}. 
\end{proof}

\begin{defn}\label{def:outer weights} (Outer Jordan--H\"older factor of Deligne--Lusztig representations) We say that a Serre weight $\sigma$ is an \emph{outer} Jordan--H\"older factor of the reduction $\ovl{R}$ of a Deligne--Lusztig representation if there exists $w,s \in W$, $\omega-\eta\in X^*(T)$, and after fixing a choice of $w^\diamond\in \tld{W}$ we have
\begin{itemize}
\item $\sigma \defeq F_{(w^\diamond,\omega)}$ and $\omega-\eta$ is $d_\sigma$-deep in $C_0$.
\item $R=R_s(\omega-s(\tld{w}_h w^\diamond)^{-1}(0))$.
\end{itemize}
Note that by \cite[Theorem 5.4(1)]{DLR}, in the current situation $\sigma$ is a Jordan--H\"older factor of $\ovl{R}$ with multiplicity one for any choice of $s\in W$. We say that $\sigma$ is  the \emph{outer} Jordan--H\"older factor of $\ovl{R}$ corresponding to $w\in W$, and denote by $\JH_\out(\ovl{R})$ to be the set of outer Jordan--H\"older factors of $\ovl{R}$.
\end{defn}

\begin{rmk}\label{rmk:outer}
\begin{enumerate}
\item \label{item:outerLAP}
The existence of $w,s\in W$ such that $R = R_s(\omega - s(\tld{w}_h w^\diamond)^{-1}(0))$ is equivalent for any choice of lowest alcove presentation $(w^\diamond,\omega)$ of a $d_\sigma$-deep Serre weight $\sigma$. 
Indeed, if $(w^\diamond \tld{\delta}^{-1},\pi^{-1}(\tld{\delta})(\omega-\eta)+\eta)$ is another lowest alcove presentation where $\tld{\delta} \in \Omega$, then $R = R_s(\omega-s(\tld{w}_hw^\diamond)^{-1}(0))$ implies that 
\begin{align*}
R &= R_s(\omega-s(\tld{w}_hw^\diamond)^{-1}(0)) \\
&= R_{\pi^{-1}(\delta)s\delta^{-1}}(\pi^{-1}(\tld{\delta})\cdot (\omega-\eta)+\eta - \pi^{-1}(\delta) s (\tld{w}_hw^\diamond)^{-1}(0) - \pi^{-1}(\delta)s\delta^{-1}\tld{\delta}(0)) \\
&= R_{\pi^{-1}(\delta)s\delta^{-1}}(\pi^{-1}(\tld{\delta})\cdot (\omega-\eta)+\eta - \pi^{-1}(\delta) s \delta^{-1} ((\delta\tld{w}_hw^\diamond)^{-1}(0) - \tld{\delta}(0))) \\
&= R_{\pi^{-1}(\delta)s\delta^{-1}}(\pi^{-1}(\tld{\delta})\cdot (\omega-\eta)+\eta - \pi^{-1}(\delta) s \delta^{-1} \tld{\delta}(\tld{w}_hw^\diamond)^{-1}(0)) \\
&= R_{\pi^{-1}(\delta)s\delta^{-1}}(\pi^{-1}(\tld{\delta})\cdot (\omega-\eta)+\eta - \pi^{-1}(\delta) s \delta^{-1} (\tld{w}_hw^\diamond\tld{\delta}^{-1})^{-1}(0)) 
\end{align*}
where $\delta$ denotes the image of $\tld{\delta}$ in $W$ and the second equality follows from \cite[Lemma 4.2]{herzig-duke}. 
\item Let $\sigma = F_{(w^\diamond,\omega)}$ with $\omega -\eta$ is $d_\sigma$-deep in $C_0$. 
Then the condition $F_{(w^\diamond,\omega)}\in \JH_\out(\ovl{R}(t_\mu s))$ where $\mu-\omega+(\tld{w}_hw^\diamond)^{-1}(0)\in \Z R \cap C_0$ is equivalent to the simple combinatorial condition
\[t_\omega\in t_\mu s (\tld{w}_hw^\diamond)^{-1}W. \]
Indeed, suppose that $F_{(w^\diamond,\omega)}\in \JH_\out(\ovl{R}(t_\mu s))$. 
After possibly changing the pair $(w^\diamond,\omega)$ using the previous remark, we have $R_s(\mu) \cong R_{s'}(\omega-s'(\tld{w}_hw^\diamond)^{-1}(0))$. 
Using that $\mu-\omega+(\tld{w}_hw^\diamond)^{-1}(0)\in \Z R \cap C_0$, Lemma \ref{lemma:0gen} implies that $s = s'$ and $\mu = \omega-s(\tld{w}_hw^\diamond)^{-1}(0)$. 
\item By Proposition \ref{prop:combinatorial membership of DL}, if $\tld{s}\in \tld{W}$ such that $\tld{s}(0)-\eta$ is $h_\eta$-deep in $C_0$, then for each $w\in W$, $R(\tld{s})$ has a corresponding outer Jordan--H\"older factor $F_{(w^\diamond,\tld{s}(\tld{w}_h w^\diamond)^{-1}(0))}$ (which is independent of the choice of $w^\diamond$). 
The other direction is clear. 
\end{enumerate}
\end{rmk}
\begin{defn}(The covering order on Serre weights)
Let $\kappa, \sigma$ be Serre weights. Assume $\kappa$ is $(h_\eta+d_\kappa)$-deep.
We say that $\kappa$ \emph{covers} $\sigma$ if $\sigma \in \JH(\ovl{R})$ for all Deligne--Lusztig $\ovl{R}$ of which $\kappa$ is an outer Jordan--H\"older factor. 
\end{defn}
\begin{warning} This definition differs from that of \cite[Definition 2.3.10]{MLM} in two aspects:
\begin{itemize}
\item We only let $R$ run over the subset of Deligne--Lusztig representations that contain $\kappa$ as an \emph{outer} Jordan--H\"older factor. 
\item We relax the genericity requirements on $\sigma, \kappa$. 
\end{itemize}
\end{warning}
As in \cite[Proposition 2.3.12]{MLM}, there are several useful alternative characterizations of the covering relation:
\begin{prop} \label{prop:covering characterization}
Let $\kappa$ and $\sigma$ be Serre weights. 
Assume that $\kappa$ is $(h_\eta+d_\kappa)$-deep and that $\kappa$ covers $\sigma$. 
Choose a presentation $\kappa=F_{(\tld{w},\omega)}$. 
Then we have the following. 
\begin{enumerate}
\item \label{item:vertexcover}
There is a presentation $\sigma=F_{(\tld{w}',\omega')}$ such that $t_{\omega'}\tld{W}_{\leq w_0\tld{w}'}\subset t_{\omega}\tld{W}_{\leq w_0\tld{w}}$. (In particular the presentations are compatible in the sense of \cite[\S 2.2]{MLM}.) 
\item Suppose that $\kappa$ is $(d_\kappa+3)$-deep (or $(d_\kappa+1)$-deep in type $A$). 
Then for any Deligne--Lusztig representation $R = R_s(\mu)$ with $(s,\mu)$ maximally split and $\kappa\in \JH(\ovl{R})$, $R$ is $h_\eta$-generic and $\sigma\in \JH(\ovl{R})$. 
\end{enumerate}
\end{prop}
\begin{proof} The first item follows by modifying the proof of the implications $(1)\implies (2)\implies (3)$ in \cite[Proposition 2.3.12]{MLM} as we now explain. 
First, the implication $(2)\implies (3)$ does not use any genericity assumption. 
Second, the implication $(1)\implies (2)$ only uses Deligne--Lusztig representations that contain $\kappa$ as an outer Jordan--H\"older factor and that such representations follow Jantzen's generic decomposition pattern (which is the reason for the genericity hypothesis in \emph{loc.~cit.}). 
Our assumptions guarantee that the types that we encounter are of the form $R_s(\mu)$ where $\mu = \omega - s(\tld{w}_h\tld{w})^{-1}(0)$. 
Since $\kappa$ is $(h_\eta+d_\kappa)$-deep, $\mu-\eta$ is $h_\eta$-deep so that \cite[Theorem 4.2]{DLR} guarantees that Jantzen's generic decomposition pattern holds for these $R_s(\mu)$. 

For the second item, the fact that $R$ is $h_\eta$-generic follows from Lemma \ref{lem:DL containing decent weight}, while the fact that $\sigma\in\JH(\ovl{R})$ follows from the first item and Proposition \ref{prop:combinatorial membership of DL}.
\end{proof}
\begin{prop}(Outer weights are isolated under covering)\label{prop:isolating}
Let $R$ be an $h_\eta$-generic Deligne--Lusztig representation and $\kappa,\sigma\in \JH(\ovl{R})$. Assume that
\begin{itemize}
\item $\kappa $ is $(h_\eta+d_\kappa)$-deep; 
\item $\sigma$ is an outer Jordan--H\"older factor of $\ovl{R}$; and 
\item $\kappa$ covers $\sigma$. 
\end{itemize}
Then $\kappa = \sigma$.  
\end{prop}
\begin{proof} 
Let $R \defeq R(t_\mu s )$ where $\mu-\eta$ is $h_\eta$-deep in $C_0$. 
The second bullet point (and Definition \ref{def:outer weights}) and Remark \ref{rmk:outer}\eqref{item:outerLAP} implies that $\sigma=F_{(w^\diamond,t_\mu s (\tld{w}_hw^\diamond)^{-1}(0))}$ for some $w\in W$. 
In particular, $\sigma$ is $0$-deep. 
Since $w_0 w^\diamond$ is a reduced factorization by \cite[Lemma 4.1.9]{LLL}, we have that $W\tld{W}_{\leq w_0w^\diamond}=\tld{W}_{\leq w_0w^\diamond}$. 
We see that
\[t_{t_\mu s(\tld{w}_hw^\diamond)^{-1}(0)}\tld{W}_{\leq w_0 w^\diamond}=t_\mu s (\tld{w}_hw^{\diamond})^{-1}\tld{W}_{\leq w_0 w^\diamond}\]
contains $t_\mu s (\tld{w}_hw^\diamond)^{-1}w_0w^\diamond=t_\mu s t_{w^{-1}(\eta)}$.

Now let $\kappa=F_{(\tld{w},\omega)}$ be a compatible lowest alcove presentation of $\kappa$ in the sense of \cite[\S 2.2]{MLM}. 
By Proposition \ref{prop:covering characterization}\eqref{item:vertexcover}, the third bullet point implies
\[t_{t_\mu s(\tld{w}_hw^\diamond)^{-1}(0)}\tld{W}_{\leq w_0 w^\diamond}\subset t_\omega \tld{W}_{\leq w_0\tld{w}}\]
(By \cite[Lemma 2.2.4]{MLM}, there is a unique lowest alcove presentation of $\sigma$ compatible with the given presentation of $\kappa$ in the sense of \cite[\S 2.2]{MLM}. 
Thus, the lowest alcove presentation obtained from Proposition \ref{prop:covering characterization}\eqref{item:vertexcover} coincides with $(w^\diamond,t_\mu s (\tld{w}_hw^\diamond)^{-1}(0))$.) 
Using the previous paragraph, we have 
\begin{equation}\label{eqn:extpoint}
t_\mu s t_{w^{-1}(\eta)}\in t_\omega \tld{W}_{\leq w_0\tld{w}}.
\end{equation}

On the other hand, by Proposition \ref{prop:combinatorial membership of DL}\eqref{item:uparrowreflect}, the fact that $\kappa\in\JH(\ovl{R}(t_\mu s))$ implies that $t_\omega \in t_\mu s \tld{u}^{-1}W$ for some $\tld{u}\in \tld{W}^+$ such that $\tld{u}\uparrow \tld{w}_h\tld{w}$. 
Multiplying both sides by $\tld{W}_{\leq w_0 \tld{w}}$ on the right and recalling from earlier in the proof that $W\tld{W}_{\leq w_0 \tld{w}} = \tld{W}_{\leq w_0 \tld{w}}$, we get 
\begin{equation}\label{eqn:compmatch}
t_\omega \tld{W}_{\leq w_0\tld{w}}=t_\mu s \tld{u}^{-1} \tld{W}_{\leq w_0 \tld{w}}. 
\end{equation}

Putting \eqref{eqn:extpoint} and \eqref{eqn:compmatch} together we learn that
\[t_{w^{-1}(\eta)}=\tld{u}^{-1}\tld{v}\]
for some $\tld{v}\leq w_0\tld{w}$. 
But since $t_{w^{-1}(\eta)}=\tld{u}^{-1}\tld{v}\leq (\tld{w}_h\tld{w})^{-1}w_0\tld{w}=\tld{w}^{-1}t_\eta\tld{w}$ (where the inequality is obtained from $\tld{u} \leq \tld{w}_h\tld{w}$ from \cite[Lemma 4.1.1]{LLL} and $\tld{v} \leq w_0\tld{w}$) and $\ell(t_{w^{-1}(\eta)}) = \ell(\tld{w}^{-1}t_\eta\tld{w})$, we see that $t_{w^{-1}(\eta)} = \tld{u}^{-1}\tld{v} = (\tld{w}_h\tld{w})^{-1}w_0\tld{w} = \tld{w}^{-1}t_\eta\tld{w}$ which implies that $\tld{u} = \tld{w}_h\tld{w}$ and $\tld{v} = w_0\tld{w}$. 
We conclude that 
\begin{itemize}
\item $w$ is the the image of $\tld{w}$ in $W$ and without loss of generality (in particular without changing compatibility) we can and do take $\tld{w}=w^\diamond$; and 
\item $\tld{u}=\tld{w}_h\tld{w}=\tld{w}_hw^\diamond$.
\end{itemize}
Thus $\kappa=F_{(\tld{w},\omega)}=F_{(w^\diamond,t_\mu s (\tld{w}_hw^\diamond)^{-1}(0))}=\sigma$.
\end{proof}

\subsection{$L$-parameters and inertial $L$-parameters}\label{subsec:L parameters}

Suppose that $H$ is a split connected reductive group over $\Z_p$ with connected center and simply connected derived subgroup.
Let $\cO_p$ be a finite \'etale $\Z_p$-algebra, which is necessarily isomorphic to a product $\prod\limits_{v\in S_p} \cO_{v}$ where $S_p$ is a finite set and $\cO_{v}$ is the ring of integers of a finite unramified extension $F_{v}$ of $\Q_p$. Assume that the coefficient ring $\cO$ contains the image of any ring homomorphism $\cO_p \ra \ovl{\Z}_p$.
Let $G_0 = \Res_{\cO_p/\Z_p} (H_{/\cO_p})$ and set $G=G_0\otimes_{\Z_p}\cO$. We make a choice of Borel and maximal torus of $H$, which induces a Borel $B$ and maximal torus $T$ of $G$, and thus make sense of the notations in section \ref{sec:affine Weyl notation}.
We warn the reader that our $H$ and $G$ here are called $G$ and $\underline{G}$ in \cite[\S 1.9.1]{MLM}. We have opted for this choice to lighten the load on notations.

We have the dual group $G^\vee_{/\cO}\defeq \prod_{F_p\ra E}H^{\vee}_{/\cO}$ and the $L$-group ${}^L G\defeq G^\vee\rtimes\Gal(E/\Qp)$ of $G_0\otimes_{\Zp}\Qp$ (here $\Gal(E/\Qp)$ acts on the set $\{F_p\ra E\}$ by post-composition).

An $L$-parameter over $E$ is a $G^\vee(E)$-conjugacy class of $L$-homomorphisms, i.e.~of continuous homomorphisms $\rho:G_{\Qp}\ra{}^L G(E)$ compatible with the projection to $\Gal(E/\Qp)$. An inertial $L$-parameter is a $G^\vee(E)$-conjugacy class of homomorphisms $\tau:I_{\Qp}\ra G^\vee(E)$ with open kernel which admit extensions to $L$-homomorphisms $G_{\Qp}\ra{}^L G(E)$. We have similar notions of (inertial) $L$-parameters over $\cO$ and $\F$.

An inertial $L$-parameter is called \emph{tame} if its restriction to the wild inertia subgroup of $I_{\Qp}$ is trivial. For $(s,\mu) \in W\times X^*(T)$, we have a tame inertial $L$-parameter over $E$ (resp.~over $\F$) $\tau(w,\mu):I_{\Q_p} \ra T^\vee(E)$ (resp.~$\taubar(s,\mu):I_{\Q_p} \ra T^\vee(\F)$) defined by the formulas in \cite[\S 2.4]{MLM}. 
All tame inertial $L$-parameters arise this way from some (in fact many) maximally split $(s,\mu)$.
Let $\tau(t_\mu s) \defeq \tau(s,\mu)$ and $\taubar(t_\mu s) \defeq \taubar(s,\mu)$. 
We say that a tame inertial $L$-parameter $\tau$ over $E$ (resp.~tame inertial $L$-parameter $\taubar$ over $\F$) is $m$-generic if $\tau \cong \tau(\tld{w})$ (resp.~$\taubar \cong \taubar(\tld{w})$) for some $\tld{w} \in \tld{W}$ with $\tld{w}(0) - \eta$ $m$-deep in $C_0$. 
\subsubsection{Herzig's set of predicted weights}
Let $\mathcal{R}$ denote the bijection on the set of Serre weights with $p$-regular highest weight which takes $F(\lambda)$ to $F(\tld{w}_h \cdot \lambda)$. 

\begin{defn}(\cite[Definition 9.2.5]{GHS})
Let $(s,\mu)$ be maximally split and $\taubar \defeq \taubar(s,\mu)$. 
Let 
\[
W^?(\taubar(s,\mu)) =  \mathcal{R}(\JH(\ovl{R}(t_\mu s))) \defeq \{\mathcal{R}(\sigma) \mid \sigma \in \JH(\ovl{R}(t_\mu s)) \textrm{ and $\sigma$ is $p$-regular}\}. 
\]
\end{defn}

If $\rhobar: G_{\Qp} \ra {}^L G(\F)$ has the property that $\rhobar|_{I_{\Qp}}$ is conjugate to $\taubar(s,\mu)$ (in particular, $\rhobar$ is tamely ramified), then we define $W^?(\rhobar)$ to be $W^?(\taubar(s,\mu))$. 
We also say that $\rhobar: G_{\Qp} \ra G^\vee(\F)$ is $m$-generic if $\rhobar^{\mathrm{ss}}|_{I_{\Qp}}$ is $m$-generic. 

Recall the following convenient characterization of $W^?(\taubar)$. 
\begin{prop}\label{prop:characteriztion of Herzig set} 
Assume that $\taubar\cong \ovl{\tau}(\tld{w}(\taubar))$ where $\tld{w}(\taubar)(0)-\eta$ is $h_\eta$-deep in $C_0$. Then $\sigma\in W^?(\taubar)$ if and only if there is a presentation $\sigma=F_{(\tld{w},\omega)}$ such that 
\[\tld{w}(\taubar)\in t_{\omega}\tld{W}_{\leq w_0\tld{w}}.\]
\end{prop}
\begin{proof} 
We can write $W^?(\taubar)=\cR(\JH(\ovl{R}(t_\mu s)))$ where $\mu-\eta$ is $h_\eta$-deep in $C_0$. Then by Proposition \ref{prop:combinatorial membership of DL}, $\cR^{-1}(\sigma) = \kappa\in \cR^{-1}(W^?(\taubar))$ if and only if there is $\tld{w}\in \tld{W}_1$, $\omega\in \eta+C_0$, and $\tld{u}\in\tld{W}^+$ such that
\begin{itemize}
\item $\kappa=F_{(\tld{w}_h^{-1}\tld{w},\omega)}$;
\item $\tld{u}\uparrow \tld{w}_h(\tld{w}_h)^{-1}\tld{w}=\tld{w}$, or equivalently $\tld{u}\leq \tld{w}$; and
\item $t_\omega \in t_\mu s \tld{u}^{-1}W$.
\end{itemize}
But the last item is equivalent to $\tld{w}(\taubar)=t_\mu s \in t_\omega W \tld{u}$, which in turn is equivalent to $\tld{w}(\taubar)\in t_{\omega}\tld{W}_{\leq w_0\tld{w}}$ by the second item (using that $w'\tld{w}$ is a reduced factorization for any $w'\in W$ by \cite[Lemma 2.2.1]{OBW}).
\end{proof}

\begin{defn}\label{defn:extremal}(Extremal weights)
For $\taubar\cong \ovl{\tau}(\tld{w}(\taubar))$ where $\tld{w}(\taubar)(0)-\eta$ is $h_\eta$-deep in $C_0$, we say that $\sigma\in W^?(\taubar)$ is \emph{extremal} (or \emph{obvious}) if and only if there is a presentation $\sigma=F_{(\tld{w},\omega)}$ such that $\tld{w}(\taubar) \in t_\omega W\tld{w}$. 

We denote the subset of extremal weights in $W^?(\taubar)$ by $W_\obv(\taubar)$. 
\end{defn}

\begin{rmk}\label{rmk:extremal}
Using Proposition \ref{prop:characteriztion of Herzig set} and the fact that $\tld{W}_{\leq w_0\tld{w}} = W\tld{w}$ for $\tld{w} \in \Omega$ (since $w_0\tld{w}$ is a reduced factorization by \cite[Lemma 4.1.9]{LLL}, $\tld{W}_{\leq w_0} = W$, and $\tld{w} \in \Omega$ is minimal with respect to the Bruhat order $\leq$), with $\taubar$ as in Definition \ref{defn:extremal}, if $\sigma=F_{(\tld{w},\omega)} \in W^?(\taubar)$ with $\tld{w} \in \Omega$, then $\sigma \in W_\obv(\taubar)$. 
\end{rmk}

The following lemma, which is a strengthening of \cite[Lemma 4.1.10]{LLL}, is the combinatorial basis for weight elimination results.
\begin{lemma} \label{lem:WE combinatorics}
Let $\taubar\cong \ovl{\tau}(\tld{w}(\taubar))$ where $\tld{w}(\taubar)(0)-\eta$ $h_\eta$-deep in $C_0$. 
Let $\sigma\notin W^?(\taubar)$ be a $d_\sigma$-deep Serre weight. 
Then there exists a $0$-generic Deligne--Lusztig representation $R$ such that 
\begin{itemize}
\item $\sigma \in \JH_\out(\ovl{R})$; and 
\item $\tld{w}(\taubar)\notin t_\nu s \Adm(\eta)$ for any $(s,\nu)$ such that $R \cong R(t_\nu s)$ and $\nu-\eta \in C_0$. 
\end{itemize} 
\end{lemma}
\begin{proof} The proof is the same as that of \cite[Lemma 4.1.10]{LLL} once we upgrade Jantzen's generic decomposition pattern to the results of \cite{DLR}. We include the proof in the more concise notation of the present paper for the reader's convenience. 

First, we claim that there is a $0$-generic Deligne--Lusztig representation $R$ such that  $\sigma \in \JH_\out(\ovl{R})$. 
Fix a presentation $\sigma=F_{(w^\diamond,\omega)}$. 
For any $u\in W$, let $\nu_u \defeq \omega-u\pi(\tld{w}_h w^\diamond)^{-1}(0)$ and $R_u \defeq R(t_{\nu_u}u)$. 
Note that $\sigma$ is $d_\sigma$-deep implies that $\nu_u-\eta$ is $0$-deep in $C_0$. 
Moreover, $\sigma \in \JH_\out(\ovl{R}_u)$. 
We note here that in fact $\{R_u\mid s\in W\}$ is precisely the set of $0$-generic Deligne--Lusztig representations $R$ such that $\sigma \in \JH_\out(\ovl{R})$ by Remark \ref{rmk:outer}\eqref{item:outerLAP}. 

Suppose the conclusion of the lemma does not hold. 
Then by the previous paragraph, for every $u\in W$, there exists $(s,\nu)$ such that $\nu-\eta \in C_0$, $R_u \cong R(t_\nu s)$, and $\tld{w}(\taubar)\in t_\nu s \Adm(\eta)$. 
This in particular implies that there is a lowest alcove presentation of $\sigma$ compatible with $\tld{w}(\taubar)$, and we assume without loss of generality that our fixed presentation is. 
Then $t_{\nu_u}u$ and $t_\nu s$ are compatible presentations of $R_u$. 
Lemma \ref{lemma:0gen} implies that $u = s$ and $\nu_u = \nu$. 

From the previous paragraph, we have that $t_\omega=t_{\nu_u} u  (\tld{w}_hw^\diamond)^{-1}ww_0u^{-1}$ for each $u\in W$ so that 
\[(t_{\nu_u} u)^{-1}\tld{w}(\taubar)=(\tld{w}w^\diamond)^{-1}ww_0u^{-1}t_{-\omega}\tld{w}(\taubar)\in \Adm(\eta).\]
We now choose $u\in W$ so that $ww_0u^{-1}t_{-\omega}\tld{w}(\taubar)=w_0\tld{s}$ with $\tld{s}\in \tld{W}^+$, so that $(\tld{w}_hw^\diamond)^{-1}w_0 \tld{s}\in \Adm(\eta)$. But then by \cite[Proposition 2.1.6]{MLM}, we have 
\[\tld{s}\leq \tld{w}_h^{-1}\tld{w}_hw^\diamond=w^\diamond\]
Since $\tld{s}\in Wt_{-\omega}\tld{w}(\taubar)$ this implies $t_{\omega}\tld{w}(\taubar)\in \tld{W}_{\leq w_0w^\diamond}$. But Proposition \ref{prop:characteriztion of Herzig set} now implies 
$\sigma\in W^?(\taubar)$, a contradiction.
\end{proof}

\subsubsection{Moduli of local $L$-parameters}
To simplify notation and aid the reader with references to the literature, we will take $\cO_p=\cO_K$, the ring of integers of a finite unramified extension $K/\Qp$ and $H = \GL_n$ so that $G_0=\Res_{\cO_K/\Z_p} \GL_n$. 
As noted in \cite[\S 1.9.2]{MLM}, in this setting an $L$-homomorphism over $E$ (resp. over $\F$) is equivalent to the notion of a Galois representation $\rho:G_K\ra \GL_n(E)$ (resp. $\rhobar: G_K\ra \GL_n(\F)$). Similarly, an inertial $L$-parameter over $E$ (resp. over $\F$) is equivalent to the more familiar notion of an inertial type, i.e. a conjugacy class of homomorphisms $I_K \ra \GL_n(E)$ (resp. $I_K \ra \GL_n(\F)$) with open kernels which admit extensions to homomorphisms $W_K \ra \GL_n(E)$ (resp. $W_K \ra \GL_n(\F)$).

In \cite{EG}, Emerton-Gee constructs a Noetherian formal algebraic stack $\mathcal{X} \defeq \mathcal{X}_{K,n}$ over $\Spf \cO$ which parametrizes rank $n$ \'{e}tale $(\varphi,\Gamma)$-modules for $K$ (see \cite[Theorem 1.2.1]{EG}). Informally, the stack $\cX$ is the correct $\ell=p$ version of the moduli spaces of Langlands parameter valued in the Langlands dual group of $G_0$.
The stack $\cX$ has the following properties (see \cite[Theorems 4.8.12 and 6.5.1]{EG}): 
\begin{itemize}
\item 
The finite type points $\cX(\ovl{\F})$ are in bijection with isomorphism classes of Galois representations $\rhobar:G_K\rightarrow \GL_n(\ovl{\F})$. 
\item  The irreducible components of the underlying reduced stack $\mathcal{X}_{\red}$ are indexed by Serre weights for $G_0(\F_p)$. 
\item For a character of $T$ (or equivalently, a cocharacter of $T^\vee$) and an inertial type $\tau$, there is a $p$-adic formal substack $\mathcal{X}^{\lambda,\tau} \subset \mathcal{X}$ parametrizing potentially semistable representations $G_K$ of Hodge--Tate weight $\lambda$ and inertial type $\tau$. If $\rhobar \in \cX^{\lambda,\tau}(\F)$, the potentially semistable Galois deformation ring $R_\rhobar^{\lambda,\tau}$ constructed by Kisin \cite{KisinPss} is a versal ring to $\cX^{\lambda,\tau}$ at $\rhobar$. 
\end{itemize}
Note that in \cite{MLM}, the notation $\cX^{\lambda,\tau}$ is used to denote the smaller stack that parametrizes only potentially crystalline as opposed to potentially semistable Galois representations. However, there is no difference between the two notions when $\tau$ is tame and $0$-generic, so we can mostly ignore this difference. 

\begin{prop}\label{prop:wtintersect}
If $\rhobar: G_{\Q_p} \ra ^L\GL_n(\F)$ is such that $\rhobar|_{I_{\Q_p}}$ is semisimple and $(n-1)$-generic and $\tau$ is an $n$-generic tame inertial type, then the following are equivalent: 
\begin{enumerate}
\item \label{item:red} $\rhobar|_{G_K} \in \cX^{\eta,\tau}$; 
\item \label{item:admissible} there exist presentations $\tld{w}(\rhobar|_{I_{\Q_p}})$ and $\tld{w}(\tau)$ such that $\tld{w}(\rhobar|_{I_{\Q_p}})\in \Adm(\eta)\tld{w}(\tau)$; 
\item \label{item:wtintersect} $\JH( \ovl{R}(t_\mu s))\cap W^?(\rhobar|_{I_{\Q_p}}) \neq \emptyset$; 
\item \label{item:extmintersect} $\JH( \ovl{R}(t_\mu s))\cap W_{\obv}(\rhobar|_{I_{\Q_p}}) \neq \emptyset$; and
\item \label{item:outintersect} $\JH_\out( \ovl{R}(t_\mu s))\cap W^?(\rhobar|_{I_{\Q_p}}) \neq \emptyset$. 
\end{enumerate}
\end{prop}
\begin{proof}
\eqref{item:red} implies \eqref{item:admissible} follows from \cite[Theorem 3.2.1]{LLL}. 
\eqref{item:admissible} implies \eqref{item:wtintersect} by \cite[Proposition 4.4.2]{LLL} with the improved descriptions of $W^?(\rhobar|_{I_{\Q_p}})$ and $\JH( \ovl{R}(t_\mu s))$ from \cite[Theorem 4.2]{DLR}. 
\eqref{item:wtintersect} implies \eqref{item:extmintersect} by \cite[Proposition 4.4.1]{LLL} (and \cite[Theorem 4.2]{DLR}). 
By a similar argument as in the proof of \cite[Proposition 4.4.1]{LLL} but taking $\omega \in X^*(T)$ so that $t_{-\omega}\tld{w}_2 \in \tld{W}_1$ and writing $\pi^{-1}(\tld{w}) = (t_{-\omega}\tld{w}_2^{-1} w' (t_{-(w')^{-1}(\omega)}\tld{w}_1)$, we see that $W^?(\rhobar|_{I_{\Q_p}})$ contains the outer weight of $\JH_\out( \ovl{R}(t_\mu s))$ corresponding to $w_0w_2\in W$ where $w_2\in W$ denotes the image of $\tld{w}_2$. 
Thus \eqref{item:wtintersect} implies that \eqref{item:outintersect} 
Clearly, either of \eqref{item:extmintersect} or \eqref{item:outintersect} implies \eqref{item:wtintersect}. 
\eqref{item:wtintersect} implies \eqref{item:red} as in the proof of \cite[Theorem 4.4.3]{LLL}. 
\end{proof}

\begin{prop}\label{prop:lower bound of stacks}
\begin{enumerate}
\item Suppose that $\tau \defeq \tau(s,\mu)$ is a tame inertial $L$-homomorphism such that $R(t_\mu s)$ is irreducible (e.g.~if $\mu-\eta$ is $1$-deep in $C_0$, by \cite[Theorem 6.8]{DeligneLusztig}), and that $\sigma \in \JH(\ovl{R}(t_\mu s))$. 
Then $\cC_\sigma \subset \mathcal{X}^{\eta,\tau}$. 
\item For any $\lambda \in X_1(T)$, $\cC_{F(\lambda)} \subset \mathcal{X}^{\eta,\tau(1,\lambda)}$. 
\end{enumerate}
\end{prop}
\begin{proof} This follows from an argument in the proof of \cite[Proposition 3.3.8]{GL3wild}, which we briefly recall. Setting $\sigma=F(\lambda)$ with $\lambda\in X_1(T)$, by \cite[Lemma 5.5.4]{EG} we can choose a Galois representation $\rhobar$ such that
\begin{itemize}
\item $\rhobar$ corresponds to a point on $\cC_\sigma$, which is the unique irreducible component of $\cX_\red$ containing it; and
\item $\rhobar$ admits an ordinary crystalline lift with Hodge-Tate weights $\lambda+\eta$. 
\end{itemize}  
By choosing an (ordinary at $p$) globalization $\ovl{r}$ of $\rhobar$, the second item implies that $\ovl{r}$ contributes to a space of algebraic automorphic form of weight $\otimes_{v| p}\sigma$, and thus also to the space of algebraic automorphic form with coefficient $\otimes_{v|p}V$ for any representation $V$ such that $\ovl{V}$ contains $\sigma$ as a Jordan-H\"older factor.

We now split the argument into two cases:
\begin{enumerate}
\item In the context of the first item, we choose $V$ to be $R_s(\mu)$. Since $R_s(\mu)$ is irreducible, it is an inertial local Langlands correspondent of $\tau(s,\mu)$ in the sense of \cite[\S 2.5]{MLM} by \cite[Proposition 2.5.5]{MLM}. Now our chosen globalization gives the existence of automorphic forms with coefficient $\otimes_{v|p}R_s(\mu)$ whose associated Galois representations lifts $\ovl{r}$. 
By local-global compatibility, the local components at $p$ of such automorphic Galois representations produces semistable lifts of $\rhobar$ of Hodge-Tate weight $\eta$ and inertial type $\tau$. This shows that $\rhobar\in\cX^{\eta,\tau}_\red$. Since $\cX^{\eta,\tau}_{\red}$ is a union of irreducible components of $\cX_\red$, it must containing $\cC_\sigma$, the unique irreducible component containing $\rhobar$.
\item In the context of the second item, we choose $V$ to be an irreducible factor of the (possibly virtual) representation $R_1(\lambda)$. By \cite[Proposition 2.5.5]{MLM}, such $V$ is an inertial local Langlands correspondent of a Weil-Deligne inertial type of the form $(\tau(1,\lambda),N)$. The same local-global compatibility argument as above then shows that $C_{F(\lambda)}$ occurs in $\cX^{\eta,\tau(1,\lambda)}$.
\end{enumerate}
\end{proof}

Proposition \ref{prop:lower bound of stacks} allows us to get a description of the underlying topological space of $\cX^{\eta,\tau}$, which strengthens \cite[Theorem 7.4.2]{MLM}:
\begin{prop} \label{prop:topological BM}Suppose $\tau=\tau(s,\mu)$ where $\mu-\eta$ is $(2h_{\eta}+2)$-deep in $C_0$. 
Then
\begin{enumerate}
\item $\cX^{\eta,\tau}_{\red}=\bigcup_{\sigma\in\JH(\ovl{R}(t_\mu s))} \cC_\sigma$
\item For $\sigma\in \JH_{\out}(\ovl{R}(t_\mu s))$, $\cX^{\eta,\tau}_\F$ is generically reduced along $\cC_\sigma$.
\end{enumerate}

\end{prop}
\begin{proof}\begin{enumerate}
\item As explained in \cite[Remark 7.4.3(3)]{MLM}, under our current assumptions we get that $\cX^{\eta,\tau}_\red\subset \bigcup_{\sigma\in\JH(\ovl{R}(t_\mu s))} \cC_\sigma$. On the other hand, Proposition \ref{prop:lower bound of stacks} shows that the $\cX^{\eta,\tau}_\red\supset \bigcup_{\sigma\in\JH(\ovl{R}(t_\mu s))} \cC_\sigma$.
\item
Combining the first item with \cite[Remark 7.4.3(3)]{MLM}, diagram \cite[{(7.17)}]{MLM} is valid under our weaker genericity assumptions. Thus a smooth cover of $\cX^{\eta,\tau}_\F$ embeds in the naive local model denoted by $\tld{M}^{\nv}_\cJ(\leq \eta,\nabla_{\textbf{a}_\tau})_\F$. By the proof of \cite[Theorem 4.6.2]{MLM}, the (maximal dimensional) irreducible components of $\tld{M}^{\nv}_\cJ(\leq \eta,\nabla_{\textbf{a}_\tau})_\F$ are parametrized by admissible pairs (cf \cite[{(2.2)}]{MLM}), which in turn are in bijection with $\JH(\ovl{R}(t_\mu s))$ by \cite[Proposition 2.3.7]{MLM}. Under this parametrization, the components $\cC_\sigma\subset \cX^{\eta,\tau}_\F$ for $\sigma\in \JH_{\out}(\ovl{R}(t_\mu s))$ correspond to the irreducible components of $\tld{M}^{\nv}_\cJ(\leq \eta,\nabla_{\textbf{a}_\tau})_\F$ parametrized by the admissible pairs $(\tld{w}_j,\tld{w}_h^{-1}\tld{w}_j)_{j\in\cJ}$ where $\tld{w}_j\in \tld{W}_1$. Hence it suffices to show $\tld{M}^{\nv}_\cJ(\leq \eta,\nabla_{\textbf{a}_\tau})_\F$ is generically reduced along such components. 

By \cite[Proposition 4.5.1]{MLM}, $\tld{M}^{\nv}_\cJ(\leq \eta,\nabla_{\textbf{a}_\tau})_\F$ is a torus torsor over $M^{\nv}_\cJ(\leq \eta,\nabla_{\textbf{a}_\tau})_\F$, and we need to show the latter is generically reduced along the irreducible components denoted by
\[\prod_{\cJ} C_{(\tld{w}_j,t_{\mu_j}s_j\tld{w}_h^{-1}\tld{w}_j)}\]
in the notation of \cite[Proposition 4.3.10]{MLM} (and taking products for $\cJ$). By definition of the latter, we see that such components are the closures of the underlying reduced scheme of 
\[M^{\nv}_\cJ(\leq \eta,\nabla_{\textbf{a}_\tau})_\F\cap \prod_{\cJ} S^0_{\F}((\tld{w}_jw_0\tld{w}_h^{-1}\tld{w}_j)^*).\]
Since $\tld{w}_jw_0\tld{w}_h^{-1}\tld{w}_j$ is of the form $t_{w_j(\eta)}$ for some $w_j\in W$, as explained at the beginning of \cite[\S 4.2]{MLM}, \cite[Theorem 9.3]{PZ} shows that the scheme theoretic intersection $M^{\nv}_\cJ(\leq \eta,\nabla_{\textbf{a}_\tau})_\F\cap \prod_{\cJ} S^0_{\F}((\tld{w}_jw_0\tld{w}_h^{-1}\tld{w}_j)^*)$ is open in $M^{\nv}_\cJ(\leq \eta,\nabla_{\textbf{a}_\tau})_\F$. Finally, (the product over $\cJ$ version) of \cite[Theorem 4.2.4]{MLM} shows that this scheme theoretic intersection is an affine space, and in particular is reduced.

\end{enumerate}
\end{proof}

Recall that we say that a complete local Noetherian $\cO$-algebra $R$ is \emph{geometrically integral} if for any finite extension $E'$ of $E$ with ring of integers $\cO'$, $R \otimes_{\cO}\cO'$ is an integral domain. 
We will require the following commutative algebra result. 

\begin{lemma}\label{lemma:Opt}
Let $R$ be a complete local Noetherian $\cO$-algebra such that 
\begin{itemize}
\item $R$ is an integral domain; and
\item there is a section $s: R \ra \cO$ such that the corresponding $E$-point of $\Spec R[1/p]$ is geometrically unibranch. 
\end{itemize} 
Then $R$ is geometrically integral. 
\end{lemma}
\begin{proof}

For a prime ideal $\mathfrak{q} \subset R$ and $n\in \N$, let $\mathfrak{q}^{(n)} \subset R$ denote the symbolic power $(\mathfrak{q}^n R_{\mathfrak{q}})\cap R$. 
Let $\fp$ denote the kernel of $s$. 
By \cite[Theorem 23(3), Chapter IV \S 12]{ZS}, $\bigcap_{n\in \N}\fp^{(n)} = 0$. 
Letting $\tld{\fp^n} = (\fp^n R[1/p])\cap R \subset \fp^{(n)}$, we see that $\bigcap_{n\in \N} \tld{\fp^n} = 0$ so that the natural map from $R$ to the completion $\widehat{R[1/p]}$ of $R[1/p]$ at the kernel of the section $R[1/p] \ra E$ induced by $s$ is injective. 

Let $E'/E$ be a finite extension with ring of integers $\cO'$. 
Then the injection from the previous paragraph induces an injection $\psi: R \otimes_{\cO} \cO' \into \widehat{R[1/p]} \otimes_E E'$. 
The codomain of $\psi$ is the completion of $R[1/p] \otimes_E E'$ at the kernel of the section $R[1/p] \otimes_E E' \ra E'$ induced by $s$. 
As $R[1/p] \otimes_E E'$ is Noetherian and excellent and this $E'$-point of $\Spec R[1/p] \otimes_E E'$ is unibranch, the codomain of $\psi$ is an integral domain \cite[\href{https://stacks.math.columbia.edu/tag/0C2E}{Tag 0C2E}]{stacks-project}. 
Thus $R \otimes_{\cO} \cO'$ is an integral domain. 
\end{proof}

The following theorem follows from the preceding lemma and results of \cite{MLM}. 

\begin{thm}\label{thm:geomintegral}
There is a polynomial $P(x_1,\ldots,x_n) \in \Z[x_1,\ldots,x_n]$, independent of $p$, such that if 
\begin{itemize}
\item $(s,\mu-\eta)$ is a lowest alcove presentation for $\tau$ and $\mu$ is $P$-generic, i.e.~for all $j\in \Hom_{\Q_p\mathrm{-alg}}(K,E)$, $p\nmid P(\mu_{j,1},\ldots,\mu_{j,n})$; and 
\item $\rhobar: G_K\ra \GL_n(\F)$ is a tame Galois representation, 
\end{itemize}
then $R_\rhobar^\tau\defeq R_\rhobar^{\eta,\tau}$ is geometrically integral or is $0$. 
\end{thm}
\begin{proof}
By Theorem \cite[Theorem 7.3.2]{MLM} the conclusion of theorem holds except possibly for the statement that we can assure $R_\rhobar^\tau$ is geometrically integral rather than just integral. 
However, the proof of \cite[Theorem 7.3.2]{MLM}, which refers to \cite[Theorem 3.7.1]{MLM} actually guarantees that $R_\rhobar^\tau$ has an $\cO$-point (see the last paragraph of the proof of \cite[Theorem 3.7.1]{MLM}). Since $R_\rhobar^\tau[\frac{1}{p}]$ is regular and hence geometrically unibranch at all its $E$-points, we conclude by Lemma \ref{lemma:Opt}. 
\end{proof}

\begin{rmk} When $\cO_p$ is a general \'{e}tale $\Zp$-algebra, an $L$-homomorphism over $\F$ corresponds to a tuple $(\rhobar_v)_{v}$ of Galois representations $\rho_v: G_{F_v}\rightarrow \GL_n(\ovl{\F})$. Hence all the results in this section apply verbatim in this more general setting, by taking the products over ($\cO$) along $S_p$ in the proofs.
\end{rmk}

\subsection{Connecting types}
The following result is the key combinatorial ingredient for our method. 
\begin{prop}\label{prop:obvweight}
Let $\alpha$ be a simple root with $h_{\omega_\alpha} = 1$. 
Let $\tld{s},\tld{w} \in \tld{W}$ such that $\tld{s}(0)-\eta$ and $\tld{w}(0)-\eta$ are $h_\eta$-deep in $C_0$. 
Suppose further that $\tld{w}^{-1}\tld{s} = \tld{w}_2^{-1}s_\alpha w_0 \tld{w}_1$ for some $\tld{w}_2 \in \tld{W}_1$ and $\tld{w}_1 \in \tld{W}^+$ with $\tld{w}_1 \uparrow \tld{w}_h^{-1} \tld{w}_2$.
Then $W^?(\taubar(\tld{s})) \cap \JH(\ovl{R}(\tld{w}))$ contains the outer weights of $\JH(\ovl{R}(\tld{w}))$ corresponding to $w_0 w_2$ and $w_0 s_\alpha w_2$, respectively (here $w_2$ is the image of $\tld{w}_2$ in $W$). 
\end{prop}
\begin{rmk}\label{rmk:connecting type} Let $\tld{s}$ be such that $\tld{s}(0)-\eta$ is $2h_\eta$-deep in $C_0$. Then given any $\sigma\in W^?(\ovl{\tau}(\tld{s}))$, we can find $\tld{w}$ such that $\sigma$ is the outer weight of $\JH(\ovl{R}(\tld{w}))$ corresponding to $w_0w_2$ in Proposition \ref{prop:obvweight}.

Indeed, taking $\tld{u}$, $\tld{w}$, and $\omega$ in the proof of Proposition \ref{prop:characteriztion of Herzig set} to be $\tld{w}_1$, $\tld{w}_h^{-1}\tld{w}_2$, and $\omega$, respectively, we can find $\tld{w}_2\in \tld{W}_1$, $\omega\in X^*(T)$ and $\tld{w}_1\in \tld{W}^+$ such that
\begin{itemize}
\item $\sigma=F_{(\tld{w}_h^{-1}\tld{w}_2,\omega)}$; 
\item $\tld{w}_1 \uparrow \tld{w}_h^{-1}\tld{w}_2$; and
\item $t_\omega\in \tld{s}\tld{w}_1^{-1}W$.
\end{itemize}  
We now choose $\tld{w}$ such that $\tld{w}^{-1}\tld{s}=\tld{w}_2^{-1}s_\alpha w_0 \tld{w}_1$. Note that $\tld{w}(0)-\eta$ is $h_\eta$-deep in $C_0$ since $\tld{s}(0) - \eta$ is $2h_\eta$-deep in $C_0$ and $\tld{w}^{-1}\tld{s}=\tld{w}_2^{-1}s_\alpha w_0 \tld{w}_1 \in \Adm(\eta)$, and  
\[t_{\omega}\in \tld{w}\tld{w}_2^{-1}s_\alpha w_0W= \tld{w}\tld{w}_2^{-1}W\]
so that $\sigma$ is the outer weight corresponding to $w_0w_2$. 
\end{rmk}

\begin{proof}
Let $w\in W$. 
The outer Jordan--H\"older factor of $R(\tld{w})$ corresponding to $w_0w$ is
$F_{(\tld{w}_h^{-1}w^\diamond,\tld{w} (w^\diamond)^{-1}(0))}$ (taking $w^\diamond$, $s$, and $\omega$ in Definition \ref{def:outer weights} to be $\tld{w}_h^{-1}w^\diamond = (w_0w)^\diamond$, the image of $\tld{w}$ in $W$, and $\tld{w}(w^\diamond)^{-1}(0)$, respectively. 
By Proposition \ref{prop:characteriztion of Herzig set}, this Serre weight belongs to $W^?(\taubar(\tld{s}))$ if and only if 
\[\tld{s}\in t_{\tld{w} (w^\diamond)^{-1}(0)}\tld{W}_{\leq w_0\tld{w}_h^{-1}w^\diamond}=\tld{w}(w^\diamond)^{-1} \tld{W}_{\leq w_0\tld{w}_h^{-1}w^\diamond},\]
or equivalently
\[w^\diamond \tld{w}^{-1}\tld{s}\leq w_0\tld{w}_h^{-1}w^\diamond. \]

Thus we need to check $w^\diamond \tld{w}_2^{-1}s_\alpha w_0 \tld{w}_1 \leq w_0\tld{w}_h^{-1}w^\diamond$ for $w\in \{w_2,s_\alpha w_2\}$. 
\begin{itemize}
\item For $w=w_2$, we can choose $w^\diamond=\tld{w}_2$. Then
\[w^\diamond \tld{w}_2^{-1}s_\alpha w_0 \tld{w}_1= s_\alpha w_0 \tld{w}_1 \leq w_0 \tld{w}_h^{-1} \tld{w}_2,\]
where the inequality follows from the fact that $\tld{w}_1\leq \tld{w}_h^{-1} \tld{w}_2$ by \cite[Theorem 4.1.1]{LLL} and that $w_0 (\tld{w}_h^{-1} \tld{w}_2)$ is a reduced factorization (see \cite[Lemma 4.1.9]{LLL}). 
\item For $w=s_\alpha w_2$, we can take $w^\diamond=(s_\alpha \tld{w}_2)^\diamond$, and the desired inequality is the content of Lemma \ref{lem:subregular shape bound}.
\end{itemize}
\end{proof}

\begin{defn}(Connected pairs of weights)\label{defn:connected weights} 
Let $\tld{s}$ be such that $\tld{s}(0)-\eta$ is $h_\eta$-deep in $C_0$, and let $\sigma,\sigma'\in W^?(\ovl{\tau}(\tld{s}))$. 
We say that $\sigma, \sigma'$ are \emph{connected} (relative to $\ovl{\tau}(\tld{s})$) if there exists a Deligne--Lusztig representation $R(\tld{w})$ and a simple root $\alpha$ satisfying the conditions of Proposition \ref{prop:obvweight}, such that $\sigma$, $\sigma'$ are the two prescribed outer weights of $\JH(\ovl{R}(\tld{w}))$. In this situation, we also say $R(\tld{w})$ connects $\sigma, \sigma'$ relative to $\ovl{\tau}(\tld{s})$.
\end{defn}

\section{Patching functors}

\subsection{The axioms}\label{sec:axioms}
We use the setup of \S \ref{subsec:L parameters}.
Thus $\cO_p$ is a finite \'etale $\Z_p$-algebra and $F_p=\cO_p \otimes_{\Z_p} \Q_p = \prod_{v\in S_p} F_v$. 
We assume that $E$ contains the image of any homomorphism $F_p \rightarrow \ovl{\Q}_p$. 
Recall that $G_0=\Res_{\cO_p/\Z_p} \GL_n$ and $\Gamma=G_0(\F_p)$. In particular $h_\eta=n-1$.

We have the dual group $G^\vee\defeq\prod_{F_p\ra E}\GL_n$ as split reductive groups over $\cO$ and the $L$-group ${}^L G\defeq G ^\vee\rtimes\Gal(E/\Qp)$ of $G_0\otimes_{\Zp} \Qp$ (where $\Gal(E/\Qp)$ acts on the set $\{F_p\ra E\}$ by post-composition).
Thus an $L$-homomorphism $\rhobar$ over $\F$ is equivalent to a tuple $(\rhobar_v)_{v\in S_p}$ of continuous representations $G_{F_v} \ra \GL_n(\F)$; while an inertial $L$-parameter $\tau$ is equivalent to a tuple $(\tau_v)_{v\in S_p}$ of inertial types i.e.~homomorphisms $\tau_v: I_{F_v} \ra \GL_n(E)$ with open kernel which admit extensions to the Weil group $W_{F_v}$. We have similar notions when replacing $E$ by $\F$.

Let $\rhobar$ be an $L$-homomorphism over $\F$. 
For an inertial $L$-parameter $\tau$, let 
\[
R_\rhobar^\tau \defeq \underset{v\in S_p,\cO}{\widehat{\bigotimes}} R_{\rhobar_v}^{\tau_v}
\]
where $R_{\rhobar_v}^{\tau_v}$ is the quotient of the universal lifting ring $R_{\rhobar_v}^\Box$ parametrizing potentially semistable lifts of Hodge--Tate weight $\eta$ and inertial type $\tau$. 
(In our applications, such potentially semistable lifts will be potentially crystalline.) 
For each $v\in S_p$, let $\Spec R_{\rhobar_v}^t$ be the reduced union of $\Spec R_{\rhobar_v}^{\tau_v} \subset \Spec R_{\rhobar_v}^\Box$ over all tame inertial types $\tau_v$. 
We set
\[
R_\rhobar^t \defeq \underset{v\in S_p,\cO}{\widehat{\bigotimes}} R_{\rhobar_v}^t.
\]
Let $R^p$ be an $\cO$-flat equidimensional complete Noetherian local $\cO$-algebra, and let $R_\infty\defeq R^p \widehat{\otimes}_{\cO} R_\rhobar^t$. 
We have the quotients
\[
R_\infty(\tau)\defeq R_\infty \widehat{\otimes}_{R_\rhobar^t}R_\rhobar^\tau. 
\]

Assume that we have the following: 
\begin{enumerate}
\item an $\cO$-algebra $S_\infty$; 
\item a perfect complex $C_\infty$ of $S_\infty[\Gamma]$-modules;
\item an $S_\infty$-algebra $\bT_\infty \subset \End_{S_\infty[\Gamma]}(C_\infty)$; and
\item \label{item:nilpotent} a surjection $R_\infty \onto \bT_\infty/I_\infty$ where $I_\infty \subset \bT_\infty$ is a nilpotent ideal. 
\end{enumerate}
For a complex of $\cO[\Gamma]$-modules $V_\bullet$, set $C_\infty(V) \defeq C_\infty \otimes^{\bL}_{\cO[\Gamma]} V_\bullet$. When $V_\bullet=V$ is a finite $\cO[\Gamma]$-module, the second item implies that $C_\infty(V)$ is a perfect complex of $S_\infty$-modules.

We assume this setup satisfies the following local-global compatibility axiom:
\begin{ax} \label{LGC} 
For any tame inertial $L$-parameter $\tau=\tau(t_\mu s)$ and any $\cO[\Gamma]$-lattice $R(t_\mu s)^\circ \subset R(t_\mu s)$ 
\[\supp_{R_\infty} H^*(C_\infty(R(t_\mu s)^\circ) \subset R_\infty(\tau).\] 
\noindent Note that we can make sense of $\supp_{R_\infty} H^*(C_\infty(V))$ even though $H^*(C_\infty(V))$ is not an $R_\infty$-module since $\Spec \bT_\infty$ (which is homeomorphic to $\Spec \bT_\infty/I_\infty$) is naturally a (closed) subset of $\Spec R_\infty$ by \eqref{item:nilpotent}. 
\end{ax}

\begin{defn}
Let $W_{C_\infty}(\rhobar)$ be the set of Serre weights $\sigma$ such that $H^d(C_\infty(\sigma)) \neq 0$ for some integer $d$. 
\end{defn}

Our first result is the generalization of \cite[Theorem 6.1]{MLM} to this context:
\begin{thm}[Weight elimination] \label{thm:WE}
If $\rhobar$ is $(2n+1)$-generic, then $W_{C_\infty}(\rhobar) \subset W^?(\rhobar^{\mathrm{ss}})$. 
\end{thm}
\begin{proof} 
Let $d$ be minimal such that there exists $\sigma\defeq F(\lambda) \notin W^?(\rhobar^{\mathrm{ss}})$ and $H^d(C_\infty(\sigma)) \neq 0$. 
First suppose that $\sigma$ is not $p$-regular, i.e. $\sigma$ is not $0$-deep.
Then $\sigma \in \JH(\ovl{R}_1(\lambda))$ and $\rhobar^{\mathrm{ss}}$, and thus $\rhobar$, does not have a potentially semistable lift of type $(\eta,\tau(1,\lambda))$ by \cite[Proposition 7]{enns} and the proof of \cite[Theorem 8]{enns} (note that $\rhobar$ is $(2n+1)$-generic in the sense of \cite[Definition 2]{enns} by \cite[Remark 2.2.8]{LLL}). 
Thus Axiom \ref{LGC} implies that $C_\infty(R_1(\lambda)^\circ) = 0$ for any lattice $R_1(\lambda)^\circ \subset R_1(\lambda)$. 
Taking a lattice with reduction $\ovl{R}_1(\lambda)$ admitting an injection $\sigma \into \ovl{R}_1(\lambda)$, the exact sequence $0 \ra \sigma \ra \ovl{R}_1(\lambda) \ra \ovl{R}_1(\lambda)/\sigma \ra 0$ shows that $H^{d-1}(C_\infty(\ovl{R}_1(\lambda)/\sigma)) \neq 0$. 
(Such a lattice can be taken to be the $\cO$-dual of the image of a nonzero homomorphism from the $\cO[\Gamma]$-projective cover of $\sigma^\vee$ to the $E$-dual of $R_1(\lambda)$.) 
By d\'evissage using a maximal filtration on $\ovl{R}_1(\lambda)/\sigma$, we see that $H^{d-1}(C_\infty(\kappa)) \neq 0$ for some $\kappa \in \JH(\ovl{R}_1(\lambda))$. 
By minimality of $d$, $\kappa \in W^?(\rhobar^{\mathrm{ss}})$ so that in particular by Proposition \ref{prop:characteriztion of Herzig set} $\kappa$ admits a presentation $(\tld{w},\omega)$ with $\tld{w}(\rhobar^{\mathrm{ss}})\in t_\omega\tld{W}_{\leq w_0\tld{w}}$. 
Since $\tld{w}(\rhobar^{\mathrm{ss}})(0)-\omega$ is then in the convex hull of $W\tld{w}(0)$, we have that $\omega - \eta$ is $(2n+1-h_{\tld{w}(0)})$-deep. 
Using that $h_{\tld{w}_h\tld{w}(0)} = h_\eta - h_{\tld{w}(0)}$ for $\tld{w} \in \tld{W}_1$, we have that $n-h_{\tld{w}(0)} \geq  d_\kappa$. 
We conclude that $\kappa$ is $(h_\eta+d_\kappa+2)$-deep.
Lemma \ref{lem:DL containing decent weight} implies that $R_1(\lambda)$ is $(h_\eta+2)$-generic (using Remark \ref{rmk:typeA} or assuming $n\geq 2$ to get that $m>2$). 
Letting $t_\mu s$ be a $(h_\eta+2)$-generic presentation of $R_1(\lambda)$, Proposition \ref{prop:combinatorial membership of DL}\eqref{item:uparrowreflect} implies that there exist $\omega-\eta \in C_0$, $\tld{u} \in \tld{W}^+$, and $\tld{w}\in \tld{W}_1$ such that $\tld{u} \uparrow \tld{w}_h\tld{w}$, $\lambda = \tld{w}\cdot(\omega-\eta)$, and $\mu-\omega \in W\tld{u}^{-1}(0)$. 
We conclude that $\lambda$ is $(h_\eta-h_{\tld{u}^{-1}(0)}+2)$-deep in $\tld{w} \cdot C_0$. 
Since $h_\eta\geq h_{\tld{w}_h\tld{w}(0)} \geq h_{\tld{u}(0)} = h_{\tld{u}^{-1}(0)}$, $\lambda$ is $2$-deep, contradicting the $p$-irregularity of $\lambda$. 

Next suppose that $\sigma$ is $0$-deep, but not $d_\sigma$-deep. 
As in the proof of \cite[Theorem 6.1]{DLR} after the sentence ``Suppose otherwise", $\sigma \in \JH(\ovl{R}_w(\nu))$ for some $w\in W$ and $\nu \in X^*(T)$ such that $\nu-\eta$ is $1$-deep but not $(d_\sigma+2)$-deep (using that $d_\sigma$ is at least $h_{\tld{w}_h\tld{w}_\lambda(0)}$ where $\sigma$ is taken to be $F(\lambda)$ in \emph{loc.~cit.}). 
We claim that $\rhobar$ does not have a potentially semistable lift of type $(\eta,\tau(w,\nu))$. 
Indeed, if it did then \cite[(6.1)]{DLR}, the sentence and parenthetical that follow (with $h_{\tld{w}_h\tld{w}_\lambda(0)}$ replaced by $d_\sigma$), and the paragraph after that imply that $\rhobar$ is not $(2n+1)$-generic which is a contradiction. 
Arguing as in the previous paragraph, there is a Serre weight $\kappa$ which is $(h_\eta+d_\kappa+2)$-deep and such that $\kappa \in \JH(\ovl{R}_w(\nu))$. 
Lemma \ref{lem:DL containing decent weight} now implies that $R_w(\nu)$ is $(h_\eta+2)$-generic which contradicts the fact that $\nu-\eta$ is not $(h_\sigma+1)$-deep. 

Finally, suppose that $\sigma$ is $d_\sigma$-deep. 
Applying Lemma \ref{lem:WE combinatorics} for a choice of $\tld{w}(\rhobar^{\mathrm{ss}})$ such that $\rhobar^{\mathrm{ss}}|_{I_{\Q_p}} \cong \ovl{\tau}(\tld{w}(\rhobar))$, there exists a $0$-generic Deligne--Lusztig representation $R$ such that 
\begin{itemize}
\item $\sigma \in \JH_\out(\ovl{R})$; and
\item $\tld{w}(\taubar)\notin t_\nu s \Adm(\eta)$ for any $(s,\nu)$ such that $R \cong R_s(\nu)$ and $\nu-\eta \in C_0$. 
\end{itemize} 
Since $\rhobar$ was assumed to be $(2n+1)$-generic, by the paragraph of the proof of \cite[Theorem 6.1]{DLR} cited above, the last item shows that $\rhobar$ does not have a potentially semistable lift of type $(\eta,\tau(s,\nu))$ (for any or equivalently all $(s,\nu)$ as above). 
Arguing as before, there is a Serre weight $\kappa \in W^?(\rhobar^{\mathrm{ss}})$ which is $(h_\eta+d_\kappa+2)$-deep and such that $\kappa \in \JH(\ovl{R}_s(\nu))$. Lemma \ref{lem:DL containing decent weight} implies that $R_s(\nu)$ is $(h_\eta+2)$-generic. 
But now applying Proposition \ref{prop:combinatorial membership of DL} to the relation $\kappa\in \JH(\ovl{R}_s(\nu))$ and Proposition \ref{prop:characteriztion of Herzig set} to the relation $\kappa\in W^?(\rhobar^{\mathrm{ss}})$, we can write $\kappa=F_{(\tld{u},\mu)}$ such that
\[\tld{w}(\rhobar)\in t_\mu \tld{W}_{\leq w_0 \tld{u}}\subset t_\nu s \Adm(\eta),\]
but this contradicts the second item above.
\end{proof}

\begin{rmk}\label{rmk:WE non-tame}(Weight elimination for non-tame $\rhobar$) For non-tame $\rhobar$, recall from \cite[Definition 3.4.1]{OBW} that one has the notion of a \emph{specalization} $\rhobar^{\mathrm{sp}}$ of $\rhobar$. We remind the reader that $\rhobar^{\mathrm{ss}}|_{I_{\Qp}}$ is one such specialization, but when $\rhobar$ is not tame there are always others.
 
We claim that for $(2n+1)$-generic $\rhobar$, each specialization $\rhobar^{\mathrm{sp}}$ has the following \emph{geometric specialization} property: there exists a family $\tld{\rho}:\bA^1_\F\rightarrow \cX$ such that $\tld{\rho}_t\cong \rhobar$ for each $t\neq 0$ while $\tld{\rho}_0|_{I_{\Qp}}\cong \rhobar^{\mathrm{sp}}$. 
Indeed, it follows from the definition in \emph{loc.cit.} that there is an $n$-generic tame type $\tau$ and a family $\fM:\bA^1_\F \rightarrow Y^{\leq\eta,\tau}$ of Breuil--Kisin modules of type $\tau$ with the above properties but for the restriction to $\prod_{v|p}G_{(F_v)_\infty}$, and that the family actually lands in $\cX^{\leq \eta,\tau}$ over $(\bG_m)_\F$ (we refer the reader to \cite[\S 3]{OBW}, \cite[\S 5]{MLM} for the definition of the stack $Y^{\leq\eta,\tau}$ of Breuil--Kisin modules of type $\tau$ as well as the meaning of the fields $(F_v)_\infty$ underlying this notion). 
Since $\rhobar$ is $(2n+1)$-generic and $\rhobar^{\mathrm{ss}}$ is $\tau$-admissible by (the proof of) \cite[Proposition 3.3.1]{OBW}, $\tau$ is actually $(h_\eta+2)$-generic by \cite[Proposition 3.1.14]{OBW} (and the fact that if $t_\nu\in W\Adm(\eta)$ then $h_\nu \leq h_\eta$). But then by \cite[Proposition 7.2.3]{MLM}, the family $\fM$ upgrades to a family $\tld{\rho}:\bA^1\rightarrow \cX^{\leq \eta,\tau}\hookrightarrow \cX$ with the desired property (indeed the restriction of $\fM$ to the open subset $\bG_m\subset \bA^1$ factors through the closed substack $\cX^{\leq \eta,\tau}\cong Y^{\leq \eta,\tau,\nabla_\infty}\subset Y^{\leq \eta,\tau}$, so all $\fM$  also factors).

Given the geometric specialization property above, the proof of Theorem \ref{thm:WE} adapts to show that for any (not necessarily tame) $(2n+1)$-generic $\rhobar$, $W_{C_\infty}(\rhobar) \subset W^?(\rhobar^{\mathrm{sp}})$ for any specialization $\rhobar^{\mathrm{sp}}$ which is $(2n+1)$-generic: indeed, the geometric specialization property shows that if $\rhobar$ occurs in a stack $\cX^{\eta,\tau}$, then we can find a tame $\rhobar'$ in $\cX^{\eta,\tau}$ such that $\rhobar'|_{I_K}\cong \rhobar^{\mathrm{sp}}$.
\end{rmk}

For a prime $\fp$ of $R_\infty$, we denote by $\bT_{\infty,\fp}$ and $C_\infty(V)_{\fp}$ the corresponding localizations if $\fp$ lies in the image of the map $\Spec \bT_\infty/I_\infty \ra \Spec R_\infty$.
Otherwise, we set $\bT_{\infty,\fp}$ and $C_\infty(V)_{\fp}$ to be $0$. 

For a Serre weight $\sigma$, the component $\cC_\sigma$ defines a radical (possibly unit) ideal in $R^t_\rhobar$, and thus by taking the extension and radical, a radical ideal $I_\sigma \subset R_\infty$. 

\begin{lemma}\label{lemma:codim}
\begin{enumerate}
\item \label{item:ht1} The height of any $I_\sigma \subset R_\infty$ is $1$. 
\item \label{item:ht>1} The height of the sum of two nonzero ideals $I_\sigma$ and $I_\kappa$ with $\sigma \neq \kappa$ is strictly larger than $1$. 
\end{enumerate}
\end{lemma}
\begin{proof}
Let $\cX^t_{\red} \subset \cX$ denote the (equidimensional) reduced union of $\cX^{\eta,\tau}_{\red}$ algebraic stack over all tame inertial types $\tau$. 
Then $(R_\rhobar^t \otimes_{\cO} \F)_{\red}$ is a versal ring for $\cX^t_{\red}$. 
If $\cC_\sigma \subset \cX^t_{\red}$, then it is an irreducible component. 
In particular, $\cC_\sigma$ and $\cX^t_{\red}$ have the same dimension. 
Letting $I'_\sigma \subset R_\rhobar^t$ denote the radical ideal corresponding to $\cC_\sigma$, we conclude from \cite[Lemma 2.40]{EG-dimension} that $\Spec R_\rhobar^t/\varpi$ (or $\Spec (R_\rhobar^t/\varpi)_{\red}$) and $\Spec R_\rhobar^t/I'_\sigma$ have the same dimension. 
This implies that $\Spec (R_\rhobar^t/\varpi) \widehat{\otimes}_{\F} (R^p /\varpi)$ and $\Spec (R_\rhobar^t/I'_\sigma)\widehat{\otimes}_{\F} (R^p /\varpi)$  (or $\Spec R_\infty/I_\sigma$) have the same dimension. 
Since the dimension of $\Spec R_\infty$ is one more than the dimension of $\Spec R_\infty/\varpi$, \eqref{item:ht1} follows. 
\eqref{item:ht>1} follows from a similar argument noting that since $\cC_\sigma$ and $\cC_\kappa$ are distinct irreducible components of $\cX$, their intersection is of strictly smaller dimension. 
\end{proof}

Since $I_\sigma$ may fail to be a prime ideal even if it is a proper ideal, we say that a prime ideal $\fp\subset R_\infty$ comes from a Serre weight $\sigma$ if $\fp$ corresponds to a minimal prime in $\Spec R_\infty/I_\sigma$. 
By Lemma \ref{lemma:codim}\eqref{item:ht1}, a prime ideal coming from a Serre weight is a minimal prime in (the equidimensional) $\Spec R_\infty/\varpi$. 
In particular, there are finitely many prime ideals of $R_\infty$ coming from any fixed Serre weight. 
Second, Lemma \ref{lemma:codim}\eqref{item:ht>1} implies that a prime ideal of $R_\infty$ comes from at most one Serre weight. Finally, we note that there exists a prime $\fp$ of $R_\infty$ coming from a Serre weight $\sigma$ if and only if $\rhobar$ occurs in the irreducible component $\cC_\sigma$ of $\cX$.

\begin{thm}[Support bound]\label{thm:supportbound}
Suppose that $\rhobar$ is $(2n+1)$-generic, and $\sigma, \kappa$ are two Serre weights. 
If a prime ideal $\fp\subset R_\infty$ comes from $\kappa$ and $C_\infty(\sigma)_{\fp} \neq 0$ (in particular $\sigma \in W^?(\rhobar^{\mathrm{ss}})$ by Theorem \ref{thm:WE} and $\sigma$ is automatically $(h_\eta+d_\sigma+2)$-deep by Proposition \ref{prop:characteriztion of Herzig set}), then $\sigma$ covers $\kappa$. 
\end{thm}
\begin{proof} 
Note that by Proposition \ref{prop:characteriztion of Herzig set}, any $\sigma\in W^?(\rhobar^{\mathrm{ss}})$ is indeed $(h_\eta+d_\sigma+2)$-deep. 
Suppose for the sake of contradiction that the theorem does not hold. 
Let $d$ be the minimal integer such that there exist a prime ideal $\fp\subset R_\infty$ that comes from $\kappa$ and a Serre weight $\sigma$ such that $H^d(C_\infty(\sigma)_{\fp}) \neq 0$ and $\sigma$ does not cover $\kappa$. 
Then there exists a (necessarily $(h_\eta+2)$-generic by Lemma \ref{lem:DL containing decent weight}) Deligne--Lusztig representation $R$ such that $\sigma \in \JH_\out(\ovl{R})$ and $\kappa \notin \JH(\ovl{R})$. 
Then a degree shifting argument as in the proof of Theorem \ref{thm:WE} shows that $H^{d-1}(C_\infty(\sigma')_{\fp}) \neq 0$ for some $\sigma' \in \JH(\ovl{R})$. By Theorem \ref{thm:WE}, $\sigma' \in W^?(\rhobar)$ so that $\sigma'$ is $(h_\eta+d_{\sigma'}+2)$-deep, and the minimality of $d$ implies that $\sigma'$ covers $\kappa$. 
But then Proposition \ref{prop:covering characterization} implies that $\kappa \in \JH(\ovl{R})$ which is a contradiction. 
\end{proof}

\begin{lemma}
Let $V$ be a finite length $\cO[\Gamma]$-module and $\fp\subset R_\infty$ a prime ideal that comes from a Serre weight. Then $H^d(C_\infty(V)_{\fp})$ has finite length over $\bT_{\infty,\fp}$ for each integer $d$.
\end{lemma}
\begin{proof} It suffices to treat the case $\varpi V=0$. But then $H^d(C_\infty(V)_{\fp})$ is a finite module over $\bT_{\infty,\fp}/\varpi$, which is a Noetherian ring of Krull dimension at most $\dim R_{\infty,\fp}/\varpi=0$, hence must be of finite length.
\end{proof}
Thus, for a prime ideal $\fp\subset \bT_\infty$ that comes from a Serre weight and a finite length $\cO[\Gamma]$-module $V$, we can define 
\[
\lgth_{\bT_{\infty,\fp}} C_\infty(V)_{\fp} = \sum_d (-1)^d\lgth_{\bT_{\infty,\fp}} H^d(C_\infty(V)_{\fp}). 
\]
Note that the assignment $V\mapsto \lgth_{\bT_{\infty,\fp}} C_\infty(V)_{\fp}$ is additive on exact triangles in the bounded derived category $D^b_{\mathrm{fl}}(\cO[\Gamma])$ of complexes with finite length cohomology.
We can also compile all these numerical invariants into a top-dimensional cycle
\[[C_\infty(V)]=\sum_{\fp} \lgth_{\bT_{\infty,\fp}} C_\infty(V)_{\fp} \ovl{\{\fp\}}\in Z_{\mathrm{top}}(\Spec R_\infty/\varpi), \]
whose formation is also additive on exact triangles. In particular if $V$ is a finite dimensional $E$-representation of $\Gamma$ and $V^{\circ}$ is any $\Gamma$-stable $\cO$-lattice then the cycle
\[[C_\infty(\ovl{V})] \defeq [C_\infty(V^\circ/\varpi)]\]
is independent of the choice of lattice.

\begin{defn}
Let $W_{C_\infty}^+(\rhobar)$ be the set of Serre weights $\sigma$ such that $\lgth_{\bT_{\infty,\fp}} C_\infty(\sigma)_{\fp} \neq 0$ for some prime ideal $\fp\subset \bT_\infty$ which comes from $\sigma$. 
\end{defn}

\noindent Note that $W_{C_\infty}^+(\rhobar) \subset W_{C_\infty}(\rhobar)$. 
In particular, if $\rhobar$ is $(2n+1)$-generic then $W_{C_\infty}^+(\rhobar)\subset W^?(\rhobar^{\mathrm{ss}})$. 

We now wish to produce elements of $W_{C_\infty}^+(\rhobar)$.
The following Lemma is the key to our approach to give a lower bound on $W_{C_\infty}^+(\rhobar)$:
\begin{lemma}\label{lem:connecting type}
Assume $\rhobar$ is $(2n+1)$-generic. 
Let $\tau=\tau(s,\mu)$ be a tame inertial type where $\mu-\eta$ is $2n$-deep in $C_0$. 
Let $\sigma\in\JH_\out(\ovl{R}(t_\mu s))$ and $R(t_\mu s)^\circ \subset R(t_\mu s)$ be an $\cO$-lattice. 
Suppose that $\fp \subset R_\infty$ is a prime ideal that comes from $\sigma$. 
Then $\fp \in \Spec R_\infty(\tau)\subset\Spec R_\infty$ and 
\[
\lgth_{\bT_{\infty,\fp}} C_\infty(\sigma)_{\fp}=\sum_{\fQ} a_\fQ \cdot \lgth_{\bT_{\infty,\fQ}}  C_\infty(R(t_\mu s)^\circ)_{\fQ} 
\]
where the sum is over minimal primes $\fQ$ of $R_\infty(\tau)$ which are contained in $\fp$ and each $a_\fQ$ is a positive integer which depends only on $\fQ$. 

If $R^p/\varpi$ is generically reduced, then there is a unique minimal prime $\fQ$ of $R_\infty(\tau)$ which is contained in $\fp$ and 
\[\lgth_{\bT_{\infty,\fp}} C_\infty(\sigma)_{\fp}= \lgth_{\bT_{\infty,\fQ}}  C_\infty(R(t_\mu s)^\circ)_{\fQ}.\] 
\end{lemma}
\begin{proof} Note that the first item of Proposition \ref{prop:topological BM} shows that any prime $\fp$ coming from $\sigma$ indeed belongs to $\Spec R_\infty(\tau)$. 
By Theorem \ref{thm:supportbound} and Proposition \ref{prop:isolating}, we have
$C_\infty(\kappa)_{\fp}=0$ for any $\kappa \in \JH(\ovl{R}(t_\mu s))\setminus \{\sigma\}$, hence
\[\lgth_{\bT_{\infty,\fp}} C_\infty(\sigma)_{\fp}=\lgth_{\bT_{\infty,\fp}} C_\infty(R(t_\mu s)^\circ/\varpi)_{\fp}.\]

It thus suffices to show that  there are positive integers $a_\fQ>0$ such that for any complex $M_\bullet\in D^b(S_\infty)$ with a homomorphism $\bT_{\infty,\fp}\ra  \End_{D^b(S_\infty)}(M_\bullet)$ with finitely generated cohomology
\[
\lgth_{\bT_{\infty,\fp}} M_\bullet/\varpi=\sum_{\fQ} a_\fQ \cdot \lgth_{\bT_{\infty,\fQ}}  (M_\bullet)_{\fQ} 
\]
This is a slight generalization of \cite[Lemma 6.3.2 and Theorem 6.3.4]{10author}, whose proof easily adapts: by devissage, one reduces to the case $M_\bullet$ is the normalization $\tld{T}(\fQ)$ of $\bT_{\infty,\fp}/\fQ$, for which the result holds with $a_\fQ=\lgth_{\bT_{\infty,\fp}} \tld{T}(\fQ)/\varpi$.

If $R^p/\varpi$ is generically reduced, then by the second item of Proposition \ref{prop:topological BM}, $R_\infty(\tau)_{\fp}/\varpi$ is formally smooth over $R^p/\varpi$, hence $R_\infty(\tau)_{\fp}/\varpi$ is reduced and thus a field. Hence $R_\infty(\tau)_{\fp}$ is a DVR with uniformizer $\varpi$, justifying the uniqueness of $\fQ$. 
In this case, $\bT_{\infty,\fp}/\fQ\cong \Spec R_\infty(\tau)_{\fp}$ so that $\tld{T}(\fQ) = \bT_{\infty,\fp}/\fQ$ and $a_{\fQ}=1$.
\end{proof}
\begin{rmk}\label{rmk:cycle of type} When $R_\infty(\tau)$ is an integral domain with unique minimal prime $\fQ$, the proof of Lemma \ref{lem:connecting type} in fact shows that
\[\lgth_{\bT_{\infty,\fp}} C_\infty(R(t_\mu s)^\circ/\varpi)_{\fp}=\lgth_{R_\infty(\tau)_{\fp}}(\tld{R}_\infty(\tau)_{\fp}/\varpi)\cdot \lgth_{\bT_{\infty,\fQ}}  C_\infty(R(t_\mu s)^\circ)_{\fQ}\]
for any $\fp$ coming from $\sigma\in \JH(\ovl{R}_s(\mu))$ (where $\tld{R}_\infty(\tau)$ is the normalization of $R_\infty(\tau)$ in $R_\infty(\tau)[\frac{1}{p}]$). Indeed, in this situation $\bT_{\infty,\fp}/\fQ\cong \Spec R_\infty(\tau)_{\fp}$, so that its normalization coincides with $\tld{R}_\infty(\tau)_{\fp}$.

Thus setting $d_{C_\infty(\tau)}=\lgth_{\bT_{\infty,\fQ}}  C_\infty(R(t_\mu s)^\circ)_{\fQ}$ (which is independent of $\fp$), we have an equality of cycles 
\[[C_\infty(\ovl{R}(t_\mu s))]\defeq [ C_\infty(R(t_\mu s)^\circ/\varpi)]=d_{C_\infty(\tau)}[R_\infty(\tau)/\varpi]\]

\end{rmk}
\begin{cor}\label{cor:spreading modularity} 
Assume that $\rhobar^{\mathrm{sp}}$ is a specialization of $\rhobar$ which is $(2n+1)$-generic. 
Suppose $\sigma, \sigma'\in W^?(\rhobar^{\mathrm{sp}})$, and let $\tau=\tau(s,\mu)$ such that 
\begin{itemize}
\item $\mu-\eta$ is $2n$-deep in $C_0$; 
\item $R(t_\mu s)$ connects $\sigma, \sigma'$ relative to $\rhobar^{\mathrm{sp}}$ in the sense of Definition \ref{defn:connected weights};
\item $\rhobar$ occurs in $\cC_\sigma$, $\cC_\sigma'$; and 
\item $R_\rhobar^\tau$ is geometrically integral. 
\end{itemize}
Then $\sigma\in W_{C_\infty}^+(\rhobar)$ if and only if $\sigma'\in W_{C_\infty}^+(\rhobar)$.
\end{cor}
\begin{proof} 
Since the conclusion is insensitive to enlarging the coefficient field $E$, we may and do assume that $R^p/\fQ^p$ (resp.~$R^p/\fp^p$ and $R_\rhobar^\tau/\fp_p$) is geometrically integral for every minimal prime $\fQ^p$ of $R^p$ (resp.~$\fp^p \subset R^p/\varpi$ and $\fp_p \subset R_\rhobar^\tau/\varpi$). 
Let $\tau$ be as in the statement of the corollary. 
By \cite[Lemma 3.3, (5) and (6)]{BLGHT}, every minimal prime ideal of $R_\infty(\tau)$ (resp.~$R_\infty(\tau)/\varpi$) has the form $\fQ^p \widehat{\otimes}_{\cO} R_\rhobar^\tau$ (resp.~$\fp^p \widehat{\otimes}_{\F} (R_\rhobar^\tau/\varpi) + (R^p/\varpi) \widehat{\otimes}_{\F} \fp_p$) where $\fQ^p \subset R^p$ is a minimal prime (resp.~$\fp^p \subset R^p/\varpi$ and $\fp_p \subset R_\rhobar^\tau/\varpi$ are minimal primes). 
Suppose that $\fp = \fp^p \widehat{\otimes}_{\F} (R_\rhobar^\tau/\varpi) + (R^p/\varpi) \widehat{\otimes}_{\F} \fp_p \subset R_\infty(\tau)/\varpi$ is a prime ideal coming from $\sigma$. 
Since $\rhobar \in \cC_{\sigma'} \subset \cX^{\eta,\tau}$, we can choose a minimal prime ideal $\fp'_p \subset R_\rhobar^\tau/\varpi$ which contains the (proper) radical ideal in $R_\rhobar^\tau/\varpi$ corresponding to $\cC_{\sigma'}$. 
Then $\fp' \defeq \fp^p \widehat{\otimes}_{\F} (R_\rhobar^\tau/\varpi) + (R^p/\varpi) \widehat{\otimes}_{\F} \fp'_p \subset R_\infty(\tau)/\varpi$ is a prime ideal that comes from $\sigma'$. 
With this choice, a minimal prime ideal of $R_\infty(\tau)$ is contained in $\fp$ if and only if it is contained in $\fp'$. 
Applying Lemma \ref{lem:connecting type}, we get
\[\lgth_{\bT_{\infty,\fp'}} C_\infty(\sigma')_{\fp'}= \sum_\fQ a_\fQ\cdot \lgth_{\bT_{\infty,\fQ}}  C_\infty(R(t_\mu s)^\circ)_{\fQ}=\lgth_{\bT_{\infty,\fp}} C_\infty(\sigma)_{\fp}\]
for any choice $\cO$-lattice $R(t_\mu s)^\circ\subset R(t_\mu s)$. 
This equality establishes the result. 
\end{proof}

Recall the set $W_{\obv}(\rhobar)$ of extremal weights defined for a $(2n-1)$-generic $L$-homomorphism $\rhobar$ in \cite[Definition 3.7.1]{OBW}. Note that $\rhobar$ is not assumed to be tame in this definition, but when $\rhobar$ is tame this definition coincides with Definition \ref{defn:extremal}.
\begin{thm} (Modularity of extremal weights)\label{thm:extremal modularity}
Assume that $\rhobar$ is $(3n-1)$-generic and that any specialization of $\rhobar$ is $(2n+1)$-generic.
If $W_{\obv}(\rhobar)\cap W_{C_\infty}^+(\rhobar) \neq \emptyset$, then $W_{\obv}(\rhobar)\subset W_{C_\infty}^+(\rhobar)$. 
\end{thm} 
\begin{proof}
This is an improved version (for $e=1$) of \cite[Theorem 5.4.3]{OBW}.
We explain how to modify the proof of \emph{loc.cit.} to our current setting. We first recall from \cite[Definition 3.6.2]{OBW} the enhancement $SP(\rhobar)$ of $W_{\obv}(\rhobar)$. 
The set $SP(\rhobar)$ consists of pairs $(\sigma,\rhobar^{\mathrm{sp}})$ where $\rhobar^{\mathrm{sp}}$ is a specialization of $\rhobar$ and $\sigma\in W^?(\rhobar^{\mathrm{sp}})$ which satisfies the conditions of \cite[Lemma 3.6.1]{OBW}. In particular $\rhobar$ occurs in $\cC_\sigma$, and the map $(\sigma,\rhobar^{\mathrm{sp}})\mapsto \sigma$ gives a surjection $ SP(\rhobar)\twoheadrightarrow W_{\obv}(\rhobar)$ by \cite[Definition 3.7.1]{OBW}. 
We have a map $\theta_{\rhobar}: SP(\rhobar)\rightarrow W$ which is injective since $\rhobar$ is $(3n-1)$-generic by \cite[Proposition 3.6.4]{OBW} (we suppress the auxilliary choice of $\zeta$ which is implicit here).

Let $\sigma\in W_{\obv}(\rhobar)\cap W_{C_\infty}^+(\rhobar)$, which arises from some $(\sigma,\rhobar^{\mathrm{sp}})\in SP(\rhobar)$. 
For a simple root $\alpha$, let $\tau(s,\mu)$ be the tame inertial type labelled $\tau_1$ in \cite[Lemma 5.4.4]{OBW} (we take $e$ in \emph{loc.~cit.}~to be $1$). 
By \cite[Proposition 2.4.9]{OBW}, the Deligne--Lusztig representation $R(t_\mu s)$ connects $\sigma$ to some $\sigma_\alpha$ relative to $\rhobar^{\mathrm{sp}}$ ($\tld{w}_1$ and $\tld{w}_2$ in the notation of Proposition \ref{prop:obvweight} correspond to $w^\diamond$ and $\tld{w}_h \tld{w}^\diamond$ in the notation of \emph{loc.~cit.}). 
Note that since $\rhobar$ occurs in $\cC_\sigma\subset \cX^{\eta,\tau(t_\mu s)}$, we can assume $\mu-\eta$ is $2n$-deep in $C_0$. 
We claim that exactly one of the following happens: 
\begin{itemize}
\item $R_\rhobar^{\tau(t_\mu s)}/\varpi$ has a unique minimal prime $\fp$, which must come from $\sigma$. In this case, we can find a specialization $\rhobar^{\mathrm{sp}}_\alpha$ such that $(\sigma,\rhobar^{\mathrm{sp}}_\alpha)\in SP(\rhobar)$ whose image under $\theta_{\rhobar}$ is $\theta_\rhobar((\sigma_\alpha,\rhobar^{\mathrm{sp}})) s_\alpha$. 
\item $R_\rhobar^{\tau(t_\mu s)}/\varpi$ has two minimal primes $\fp, \fp_\alpha$, which come from $\sigma$ and $\sigma_\alpha$ respectively. In this case, $(\sigma_\alpha,\rhobar^{\mathrm{sp}})\in SP(\rhobar)$ whose image under $\theta_{\rhobar}$ is also $\theta_\rhobar((\sigma_\alpha,\rhobar^{\mathrm{sp}})) s_\alpha$.
\end{itemize}
Indeed, by \cite[Theorem 4.1.1]{OBW}, $R_\rhobar^{\tau(t_\mu s)}/\varpi$ has either one or two minimal primes. 
Moreover, by the second paragraph of the proof of \cite[Lemma 5.4.4]{OBW}, 
\[
\tld{w}(\rhobar,\tau(s,\mu)) \in \{t_{(s_\alpha w)^{-1}(\eta)},(\tld{w}_h w^\diamond)^{-1} w_0 s_\alpha w^\diamond\}. 
\]
If $R_\rhobar^{\tau(t_\mu s)}/\varpi$ has one minimal prime, then $\tld{w}(\rhobar,\tau(s,\mu)) = t_{(s_\alpha w)^{-1}(\eta)}$ and $\tau(s,\mu)$ defines a specialization as in the first bullet point (see the proof of \cite[Corollary 5.4.5]{OBW}). 
If there are two minimal primes, then $\tld{w}(\rhobar,\tau(s,\mu)) = (\tld{w}_h w^\diamond)^{-1} w_0 s_\alpha w^\diamond$. 
Then the tame inertial type labelled $\tau_2$ in \cite[Lemma 5.4.4]{OBW} gives a specialization as described in the second bullet point. 
One minimal prime of $R_\rhobar^{\tau(t_\mu s)}/\varpi$ comes from $\sigma$ and the other comes from $\sigma_\alpha$, the only other element in $W^?(\rhobar^{\mathrm{sp}}) \cap \JH(\ovl{R}(t_\mu s))$ by \cite[Proposition 2.4.9]{OBW}. 
Furthermore, using that $R_\rhobar^{\tau(t_\mu s)}$ is geometrically integral (for example, $R_\rhobar^{\tau(t_\mu s)}/\varpi$ is reduced and this property persists under base change), Corollary \ref{cor:spreading modularity} shows that, in the second case, $\sigma_\alpha\in W_{\obv}(\rhobar)\cap W_{C_\infty}^+(\rhobar)$. 

By repeating this process for varying $\alpha$, we simultanously get that $\theta_{\rhobar}$ is a bijection and $W_{\obv}(\rhobar)\subset W_{C_\infty}^+(\rhobar)$.
\end{proof}

\begin{thm}\label{thm:main modularity}
There is a polynomial $P(x_1,\ldots,x_n) \in \Z[x_1,\ldots,x_n]$, independent of $p$, such that if $\rhobar$ is tame and has lowest alcove presentation $(s,\mu-\eta)$ where $\mu$ is $P$-generic (i.e.~for all $j\in \Hom_{\Q_p\mathrm{-alg}}(F_p,E)$, $p\nmid P(\mu_{j,1},\ldots,\mu_{j,n})$) and $W_{C_\infty}^+(\rhobar) \neq \emptyset$, then $W_{C_\infty}^+(\rhobar)= W^?(\rhobar)$. 
\end{thm}
\begin{proof} By Theorem \ref{thm:geomintegral} and \cite[Proposition 7.4.7(2)]{MLM}, we can choose $P$ such that for $\rhobar$ satisfying the hypothesis of the theorem
\begin{itemize}
\item $\rhobar$ is $(3n-1)$-generic;
\item $\rhobar$ occurs in $\cC_\sigma$ for all $\sigma\in W^?(\rhobar)$; and
\item for each tame inertial type $\tau$, $R_\rhobar^\tau$ is either geometrically integral or is $0$.
\end{itemize}

Let $\sigma\in W_{C_\infty}^+(\rhobar)\subset W^?(\rhobar)$.
By Remark \ref{rmk:connecting type}, there is a sequence of weights $\sigma_0=\sigma, \cdots \sigma_k=F_{(1,\omega)}$ in $W^?(\rhobar)$ such that $\sigma_i, \sigma_{i+1}$ are connected. By Corollary \ref{cor:spreading modularity}, each $\sigma_i\in W_{C_\infty}^+(\rhobar)$. But $\sigma_k$ belongs to $W_\obv(\rhobar)$ by Remark \ref{rmk:extremal}. Thus $W_\obv(\rhobar)\cap  W_{C_\infty}^+(\rhobar)\neq \emptyset$, and hence by Theorem \ref{thm:extremal modularity}, $W_\obv(\rhobar)\subset W_{C_\infty}^+(\rhobar)$.
Finally for an arbitrary $\sigma'\subset W^?(\rhobar)$, Remark \ref{rmk:connecting type} again shows that there is a sequence of weights $\sigma'_0=\sigma, \cdots, \sigma'_\ell=F_{(1,\omega')}$ such that $\sigma'_i, \sigma'_{i+1}$ are connected. 
But since $\sigma'_\ell \in W_\obv(\rhobar)\subset W_{C_\infty}^+(\rhobar)$, another application of Corollary \ref{cor:spreading modularity} gives $\sigma'=\sigma'_0\in W_{C_\infty}^+(\rhobar)$. 
\end{proof}
\begin{rmk}\begin{enumerate}
\item As is used in the proof, $W^?(\rhobar)$ coincides with the set $W^g(\rhobar)$ of geometric weights of $\rhobar$, i.e. Serre weights $\sigma$ such that $\rhobar\in \cC_\sigma$.
\item Assume that $R^p$ is formally smooth over $\cO$. 
Then in the setting of Theorem \ref{thm:main modularity} with the polynomial $P$ chosen as in its proof, as in Remark \ref{rmk:cycle of type}, for all tame inertial type $\tau=\tau(t_\mu s)$, we have an equality of top-dimensional cycles
\[[ C_\infty(\ovl{R}(t_\mu s))]=d_{C_\infty(\tau)}[R_\infty(\tau)/\varpi].\] 
Furthermore, by comparing the coefficient of the (unique) irreducible cycle $\ovl{\{\fp\}}$ that comes from $\sigma \in \JH_\out( \ovl{R}(t_\mu s))\cap W^?(\rhobar)$ we see that
\begin{equation}\label{eqn:patchcycle}
 [C_\infty(\sigma)_\fp]= d_{C_\infty(\tau)}\ovl{\{\fp\}}.
\end{equation}
By using \eqref{eqn:patchcycle} for connecting types, we learn that $d_{C_\infty(\tau)}=d$ is independent of the connecting type $\tau$. 
Thus $[C_\infty(\sigma)_\fp]= d\ovl{\{\fp\}}$ for all $\sigma \in W^?(\rhobar)$. 
Using \eqref{eqn:patchcycle} for any tame inertial type with $R_{\rhobar}^\tau \neq 0$, or equivalently $\JH_\out( \ovl{R}(t_\mu s))\cap W^?(\rhobar) \neq \emptyset$ by Proposition \ref{prop:wtintersect}, we see that $d_{C_\infty(\tau)}=d$ for any tame inertial type $\tau$. 
Hence the set of cycles $[C_\infty(\sigma)]$ solves the system of Breuil-M\'ezard equations for potentially crystalline stacks with Hodge-Tate weights $\eta$ and tame inertial types $\tau$. Since they also satisfy the support condition prescribed by Theorem \ref{thm:supportbound}, they are uniquely determined, by invoking either \cite[Theorem 11.0.2]{FLH} and Remark \cite[11.0.2]{FLH} or \cite[\S 8.6.1]{MLM}. In particular, $\frac{1}{d}[C_\infty(\sigma)]$ is the pullback to $R_\infty$ of the Breuil-M\'ezard cycle $\cZ_\sigma$ associated to $\sigma$ constructed in \cite[Theorem 10.0.1]{FLH}.
\item The main reason we need the polynomial genericity in Theorem \ref{thm:main modularity} is that we invoked \cite[Theorem 7.3.2]{MLM} to guarantee that $R_\infty(\tau)$ is geometrically integral. However, for $n=3$, this last fact is already guaranteed by $\tau$ being 4-generic, cf \cite[Theorem 1.2.2]{GL3wild}. Thus Theorem \ref{thm:main modularity} holds for $n=3$ when $\rhobar$ is tame and $6$-generic. Furthermore, using the combinatorial classification of the set of \emph{geometric weights} $W^g(\rhobar)$ (\cite[Definition 1.2.1]{GL3wild}) in \cite[Theorem 4.2.5]{GL3wild}, one can still construct a sequence of connecting types as in the above proof to show that Theorem \ref{thm:main modularity} holds even for possibly wildly ramified $\rhobar$ which is $6$-generic. 
\item In fact, our proof of Theorem \ref{thm:main modularity} does not require that $R_\rhobar^\tau$ is geometrically integral for all tame types $\tau$ (which is essentially equivalent to $\cX^{\eta,\tau}$ being unibranch at $\rhobar$). Rather, the only information we really need is that whenever $\sigma$, $\sigma'\in W^?(\rhobar)$ is connected by some $R(t_\mu s)$, the primes $\fp$, $\fp'$ that come from $\sigma,\sigma'$ generize to the same set of minimal primes $\fQ$ of $R_\infty(\tau(t_\mu s))$. This weaker condition seems easier to check in practice than the unibranch-ness of $\cX^{\eta,\tau}$ at $\rhobar$: for instance, it is implied by the condition that $\rhobar$ occurs in a certain irreducible component of $\cC_\sigma\cap \cC_{\sigma'}$. This idea is picked up in \cite{GSp4}.
\end{enumerate}
\end{rmk}

\section{Homology of arithmetic locally symmetric spaces}

\subsection{Setup}
Let $F/\Q$ be a CM (or totally real) extension. 
Let $K_\infty$ be a maximal compact subgroup of $\GL_n(F\otimes_{\Q} \R)$. 
For example, we can take $K_\infty$ to be a product of $[F:\Q]$ copies of $\mathrm{O}_n(\R)$ (resp.~a product of $\frac{1}{2}[F:\Q]$ copies of $\mathrm{U}(n)$) if $F$ is totally real (resp.~totally imaginary). 
For a compact open subgroup $K \subset \GL_n(\A_F^\infty)$, let 
\[
Y(K) \defeq \GL_n(F) \backslash \GL_n(\A_F) / K K_\infty \R^\times. 
\]
If $K$ is neat, then $Y(K)$ is naturally a real manifold. 

Let $S$ be a finite set of finite places of $F$ stable under complex conjugation. 
If $K$ can be written as a product $K^S K_S$, a finite smooth $\cO[K_S]$-module $V$ defines a tautological local system on $Y(K)$ which we also denote by $V$. 
Let $R\Gamma(Y(K),V) \in \mathbf{D}(\cO)$ denote the derived global sections of the local system $V$. 

Let $G^S \defeq \GL_n(\A_F^{\infty,S})$, and let $\cH(G^S,K^S)$ denote the Hecke algebra over $\cO$ associated to the double coset space $K^S \backslash G^S / K^S$. 
Then there is a canonical homomorphism 
\begin{equation} \label{eqn:hecke}
\cH(G^S,K^S) \ra \End_{\mathbf{D}(\cO)}(R\Gamma(Y(K),V)). 
\end{equation}
When $K^S = \prod_{v\notin S,v\nmid\infty} \GL_n(\cO_{F_v})$, we write $\bT^S \defeq \cH(G^S,K^S)$ and $\bT^S(K,V)$ for the image of \eqref{eqn:hecke}. 

For each finite place $v$ of $F$, fix a uniformizer $\varpi_v$. We have
\begin{itemize}
\item for $0\leq i \leq n$, the elements 
\[
T_{v,i} \defeq [\GL_n(\cO_{F_v}) \mathrm{diag}(\varpi_v,\ldots,\varpi_v,1,\ldots,1)\GL_n(\cO_{F_v})] \in \cH(\GL_n(F_v),\GL_n(\cO_{F_v})) 
\]
where $\varpi_v$ appears $i$ times on the diagonal; and 
\item the polynomial 
\[
P_v(X) = \sum_{i=0}^n (-1)^i \mathbf{N}_{F/\Q}(v)^{\binom{i}{2}} T_{v,i}X^{n-i} \in\cH(\GL_n(F_v),\GL_n(\cO_{F_v}))[X].
\]
\end{itemize}

Suppose that $K^S = \prod_{v\notin S,v\nmid\infty} \GL_n(\cO_{F_v})$. 
For a finite place 
$v$ of $F$, recall that $\Frob_v$ denotes the geometric Frobenius element of $G_{F_v}/I_{F_v}$. 
Assuming either that $F$ contains an imaginary quadratic field or the work of Arthur \cite{arthur}, for a maximal ideal $\fm\subset \bT^S(K,V)$, Scholze \cite[Corollary V.4.3]{scholze} showed there is an associated continuous semisimple Galois representation \[
\rbar = \rbar_\fm: G_{F,S} \ra \GL_n(\bT^S(K,V)/\fm) 
\]
such that $\det(XI_n - \rbar(\Frob_v))$ is equal to the image of $P_v(X)$ in $(\bT^S(K,V)/\fm)[X]$ for all finite places $v\notin S$. 
Similarly, if $\pi$ is a regular $L$-algebraic cuspidal automorphic representation of $\GL_n(\A_F)$ such that $\pi_v$ is unramified for all finite places $v\notin S$, and we fix an isomorphism $\iota: \ovl{\Q}_p \risom \C$, then Scholze \cite[Corollary V.4.2]{scholze} showed there is an associated Galois representation
\[r_\iota(\pi): G_F\ra \GL_n(\ovl{\Q}_p)\]
characterized by a similar equality of the coefficients characteristic polynomial of $r_\iota(\pi)(\Frob_v)$ and (pullback under $\iota$ of) the system of Hecke eigenvalues of $\pi$.

Suppose that $S$ is a finite set of finite places containing all places dividing $p$. 
We say that $K \subset \GL_n(\A_F^\infty)$ is $(S,p)$-good if $K$ can be written as a product $\prod_{v\nmid \infty} K_v$ with $K_v = \GL_n(\cO_{F_v})$ for all $v \notin S$ and for all $v\mid p$. 

\begin{defn}
Let $W(\rbar_\fm)$ be the set of simple $\F[\prod_{v\mid p} \GL_n(\cO_{F_v})]$-modules $\sigma$ such that $R\Gamma(Y(K),\sigma)_\fm$ is nonzero for some $(S,p)$-good $K$. 
\end{defn}

\subsection{The main result}
Our main global theorem is the following. 

\begin{thm}\label{thm:main}
Let $F$ be an imaginary CM field with maximal totally real subfield $F^+$. 
Assume that 
\begin{itemize}
\item $p > 2n+1$ is a prime unramified in $F$; 
\item $F$ contains an imaginary quadratic field $F_0$ in which p splits; and
\item for each place $v|p$ of $F^+$, there exists a place $v' \neq v$ of $F^+$ such that 
\[
\sum_{v'' \neq v,v'} [F^+_{v''}:\Q_p] > \frac{1}{2} [F^+: \Q]. 
\]
\end{itemize}

Let $\rbar: G_F \ra \GL_n(\F)$ be a continuous representation. 
Assume that 
\begin{itemize}
\item $\rbar$ is decomposed generic in the sense of \cite[Definition 4.3.1]{10author}; 
\item $\rbar|_{G_{F(\zeta_p)}}$ is absolutely irreducible; 
\item $\rbar(G_F - G_{F(\zeta_p)})$ contains a scalar matrix; and 
\item there exists a cuspidal automorphic representation $\pi$ of $\GL_n(\A_F)$ satisfying the following conditions:
\begin{itemize}
\item $\pi$ has weight $0$; 
\item $\pi$ has Iwahori fixed vectors away from $p$; 
\item $\pi^{\ker(\GL_n(\cO_F)\ra \GL_n(\cO_F/p))} \neq 0$; and
\item $\ovl{r_\iota(\pi)} \cong \rbar$ for some isomorphism $\iota: \ovl{\Q}_p \risom \C$. 
\end{itemize}
\end{itemize}

Let $\rhobar \defeq (\rbar|_{G_{F_w}})_{w|p}$. Then
\begin{enumerate}
\item if $\rhobar$ is $(2n+1)$-generic, then $W(\rbar) \subset W^?(\rhobar^{\mathrm{ss}})$; and
\item there is a polynomial $P(x_1,\ldots,x_n) \in \Z[x_1,\ldots,x_n]$, independent of $p$, such that if $\rhobar$ is tamely ramified and $P$-generic, then $W(\rbar) = W^?(\rhobar)$. 
\end{enumerate}
\end{thm}

\begin{rmk}\mbox{}
\begin{enumerate}
\item The inclusion $W(\rbar) \subset W^?(\rhobar^{\mathrm{ss}})$ does not require $\rbar$ to arise from an automorphic representation $\pi$, but merely from a mod $p$ Hecke eigenclass (note that such classes typically do not lift). 
\item The appearance of the hypothesis on the local behavior of $\pi$ is to guarantee that the intersection of the component of each local deformation ring determined by $\pi$ and any other component does not contain an irreducible component of the special fiber. 
This property is needed for us to ensure in certain circumstances that the sum appearing in Lemma \ref{lem:connecting type} is nonzero. 
This allows us to produce elements in $W(\rbar)$. 
We expect that this hypothesis can be considerably relaxed. 
\end{enumerate}
\end{rmk}

\begin{proof}
Let $\Gamma \defeq \GL_n(\cO_F/p)$. 
The existence of $\pi$ shows that there exists a finite set $S$ of finite places of $F$ and a maximal ideal $\fm \subset \bT^S(K,\cO[\Gamma])$ for some $K$ which is $(S,p)$-good such that $\rbar \cong \rbar_\fm$. 
Let 
\[
\tld{R}^{\mathrm{loc}} \defeq \widehat{\otimes}_{w\in S,\cO} R_{\rbar|_{G_{F_w}}}^\Box 
\]
be the completed tensor product of local lifting rings at places in $S$, and define its quotient
\[
R^{\mathrm{loc}} \defeq \widehat{\otimes}_{w\in S, w\nmid p,\cO} R_{\rbar|_{G_{F_w}}}^\Box \widehat{\otimes}_\cO \widehat{\otimes}_{w|p,\cO} R_{\rbar|_{G_{F_w}}}^t. 
\]
As explained in the proof of \cite[Theorem 8.1]{MT}, the image hypotheses on $\rbar$ (see also \cite[Theorem A.9]{thorne}) show that we can perform ultrapatching for the complex $R\Gamma(X(K),\cO[\Gamma])_\fm$ (and $R\Gamma(X(K_1(Q_N)),\cO[\Gamma])_\fm$ for sets $Q_N$ of Taylor--Wiles primes) as in \cite[\S 6.4]{10author}. 
By furthermore keeping track of the $\Gamma$-symmetry as in \cite[\S 3]{GN} (especially \cite[Lemma 3.4.11]{GN}), we obtain 
\begin{enumerate}
\item a formally smooth $\tld{R}^{\mathrm{loc}}$-algebra $\tld{R}_\infty$;
\item a formally smooth $\cO$-algebra $S_\infty$ with augmentation ideal $\mathfrak{a}_\infty$; 
\item a perfect complex $C_\infty$ of $S_\infty[\Gamma]$-modules; 
\item an $S_\infty$-algebra $\bT_\infty \subset \End_{S_\infty[\Gamma]}(C_\infty)$; and 
\item a surjection $\tld{R}_\infty \onto \bT_\infty/I_\infty$ for some nilpotent ideal $I_\infty\subset \bT_\infty$. 
\end{enumerate}
This satisfies the following properties. 
\begin{enumerate}
\item $C_\infty \otimes_{S_\infty}^{\bL} S_\infty/\mathfrak{a}_\infty = R\Gamma(X(K),\cO[\Gamma])_\fm$. 
\item Let $\ell_0 = [F^+:\Q]n-1$ and $R_\infty \defeq \tld{R}_\infty \otimes_{\tld{R}^{\mathrm{loc}}} R^{\mathrm{loc}}$. Then $\dim S_\infty=\dim R_\infty+\ell_0$. 
\end{enumerate}

Our hypotheses on the field $F$ and $\rbar$ show that the local-global compatibility result \cite[Theorem 5.10]{Hevesi} applies. 
This implies that our axiomatic setup of \S \ref{sec:axioms} applies (after possibly enlarging $I_\infty$): first, the surjection $\tld{R}_\infty \onto \bT_\infty/I_\infty$ factors through $R_\infty$ and second, the local-global compatibility Axiom \ref{LGC} holds. 
Note that the control statement $C_\infty \otimes_{S_\infty}^{\bL} S_\infty/\mathfrak{a}_\infty = R\Gamma(X(K),\cO[\Gamma])_\fm$ shows that $W(\rbar) = W_{C_\infty}(\rhobar)$. 
In particular, $W(\rbar) \subset W^?(\rhobar)$. 
Let $P$ be the polynomial in Theorem \ref{thm:main modularity}. 
We will show $W(\rbar) \supset W^?(\rhobar)$ by invoking Theorem \ref{thm:main modularity} in the above setting for a subgroup $K$ such that $K_w$ is Iwahori at all places not dividing $p$ where $\pi$ is ramified (and pro-$w_1$ Iwahori at an appropriate auxiliary place $w_1$ to ensure $K$ is neat). 
With this choice of $K$, we can (for instance, by \cite[Theorem 3.1.1]{10author}) and will modify $R^{\mathrm{loc}}$ so that its factors away from $p$ are the unipotently ramified lifting rings. 

We need to show that $W_{C_\infty}^+(\rhobar) \neq \emptyset$ to apply Theorem \ref{thm:main modularity}. 
The existence of the automorphic representation $\pi$ such that $\rbar \cong \ovl{r_\iota(\pi)}$ shows that 
\[
H^*(C_\infty \otimes^{\bL}_{S_\infty} S_\infty/\mathfrak{a}_\infty)[1/p] \neq 0  
\]
and that these groups are nonzero only for degrees in the interval $[q_0,q_0+\ell_0]$ where $q_0 = [F^+:\Q]n(n-1)/2$. 
The proof of \cite[Proposition 6.3.8]{10author} shows that $\lgth_{\bT_{\infty,\fQ}}  C_\infty(\cO[\Gamma])_{\fQ} \neq 0$ for some minimal prime $\fQ \subset R_\infty$. 
This implies that $\lgth_{\bT_{\infty,\fQ}}  C_\infty(V)_{\fQ} \neq 0$ for some $\cO$-lattice $V$ in an irreducible $E[\Gamma]$-subrepresentation of a Deligne--Lusztig representation $R(t_\mu s)$. 
In particular, $\fQ$ comes from a minimal prime of $R_\infty(\tau(\mu,s))$, so this ring must be nonzero. 
Choosing $P$ to be divisible by $P_{3n}$, with $P_m$ defined as in \cite[Remark 2.1.11(2)]{MLM}, we see that $\rhobar$ is $3n$-generic so that we may assume that $\mu -\eta$ is $2n$-deep in $C_0$ and $V[1/p] \cong R(t_\mu s)$ by \cite[Proposition 3.3.2]{LLL}. 
Moreover, we can find $\sigma \in \JH_\out(R(t_\mu s))\cap W^?(\rhobar)$ by Proposition \ref{prop:wtintersect}. 
By Proposition \ref{prop:topological BM}, there exists a prime ideal $\fp$ coming from $\sigma$ containing $\fQ$. 
By our choice of local deformation rings away from $p$, $\fQ$ is the unique minimal prime of $R_\infty(\tau(\mu,s))$ contained in $\fp$ (indeed, the unipotently ramified lifting rings away from $p$ have generically reduced special fiber, see \cite[\S 3]{taylor} and \cite[\S 3]{ANT}).  
Then Lemma \ref{lem:connecting type} implies that $\lgth_{\bT_{\infty,\fp}}  C_\infty(\sigma)_{\fp} \neq 0$. 
Theorem \ref{thm:main modularity} implies that $W^?(\rhobar) \subset W_{C_\infty}^+(\rhobar) \subset W(\rbar)$. 
\end{proof}

\newpage
\bibliography{Biblio}
\bibliographystyle{amsalpha}

\end{document}